\documentclass[11pt,a4paper]{amsart}

\usepackage{
    amsfonts, amsmath, amssymb, ascmac,
    empheq, enumitem,
    hyperref,
    mathrsfs, mathtools, mleftright,
    tikz
}

\setlist[enumerate]{%
  topsep=2pt,
  leftmargin=6ex
}
\setlist[itemize]{%
  topsep=2pt,
  leftmargin=6ex
}

\usetikzlibrary{cd}% dvipdfmx

\newtheorem{thm}{Theorem}[section]
\newtheorem{cor}[thm]{Corollary}
\newtheorem{lem}[thm]{Lemma}
\newtheorem{prop}[thm]{Proposition}

\theoremstyle{definition}
\newtheorem{defn}[thm]{Definition}
\newtheorem{exam}[thm]{Example}
\newtheorem*{defn*}{Definition}
\newtheorem*{exam*}{Example}

\theoremstyle{remark}
\newtheorem{rem}[thm]{Remark}

\newcommand{\inn}[1]{\langle{#1}\rangle}

\newcommand{\longinto}{\lhook\joinrel\longrightarrow}
\newcommand{\longonto}{\relbar\joinrel\twoheadrightarrow}

\newcommand{\wlim}{\text{weak*-}\lim}

\newcommand{\C}{\ifmmode C^{\ast}\else C${}^{\ast}$\fi}
\newcommand{\W}{\ifmmode W^{\ast}\else W${}^{\ast}$\fi}
\newcommand{\I}{\mathrm{I}}
\newcommand{\II}{\mathrm{II}}
\newcommand{\III}{\mathrm{III}}

\renewcommand{\hat}{\widehat}
\renewcommand{\mod}{\operatorname{mod}}
\renewcommand{\tilde}{\widetilde}

\DeclareMathOperator{\Ad}{Ad}

\DeclareMathOperator{\Aut}{Aut}

\DeclareMathOperator{\clInt}{\overline{Int}}

\DeclareMathOperator{\diam}{diam}

\DeclareMathOperator{\Hom}{Hom}
\DeclareMathOperator{\im}{im}

\DeclareMathOperator{\Int}{Int}
\DeclareMathOperator{\id}{id}

\DeclareMathOperator{\Map}{Map}

\DeclareMathOperator{\Out}{Out}

\DeclareMathOperator{\Tr}{Tr}

\newcommand{\Obss}{\Ob_{\mathrm{ss}}}
\newcommand{\pt}{\mathrm{pt}}
\newcommand{\U}{\mathcal{U}}
\newcommand{\Z}{\mathcal{Z}}
\DeclareMathOperator{\Ob}{Ob}
\DeclareMathOperator{\Gau}{Gauge}% `\Gauge' already defined in hyperref?
\DeclareMathOperator{\Spin}{Spin}
\DeclareMathOperator{\String}{String}

\title{Compact group Rohlin actions and $G$-kernels on von Neumann algebras}
\author{Takumi Nishihara}
\address{RIMS, Kyoto University, 606-8502 Japan}
\email{nishihar@kurims.kyoto-u.ac.jp}
\subjclass[2020]{Primary~46L40; Secondary~46L55, 20J06}
\date{February 10, 2026}

\begin{document}

\begin{abstract}
    We provide a new construction of a topological group model for the string group of a compact, simple, and simply-connected Lie group, by solving the obstruction realization problem for compact group $G$-kernels on full factors.
    Furthermore, we introduce the Rohlin property for actions and cocycle actions of compact groups in order to establish cohomology vanishing theorems.
\end{abstract}

\maketitle

\renewcommand{\thethm}{\Alph{thm}}
\section*{Introduction}

Let $G$ be a locally compact group.
A $G$-kernel on a factor $M$ is a group homomorphism $\kappa$ from a locally compact group $G$ into $\Out(M)$ that admits a measurable lift $\alpha\colon G\to\Aut(M)$.
Given a $G$-kernel $\kappa$, whether a lift $\alpha$ can be chosen as a genuine action is a classical problem.
For each $G$-kernel $\kappa$, there is an associated 3-cohomology class $\Ob(\kappa)$ in a measurable group cohomology $\mathrm{H}^3(G;\mathbb{T})$ which serves as the obstruction to lifting $\kappa$ to a genuine action.
Sutherland \cite{Su80} has shown that this class completely determines the obstruction to a lifting action when $M$ is properly infinite.
In this context, we are interested in the following questions.
\begin{enumerate}
    \item Which 3-cohomology classes in $\mathrm{H}^3(G;\mathbb{T})$ is realized as the obstruction for a $G$-kernel?
    \item Does every $G$-kernel on a finite factor with trivial obstruction class lift to an action?
\end{enumerate}

We first address the realization problem of a 3-cohomology class in $\mathrm{H}^3(G;\mathbb{T})$ as an obstruction.
For discrete groups, Jones \cite{Jo79} has shown that every 3-cohomology class is realized as the obstruction for a $G$-kernel on the AFD $\II_1$ factor.
% finite group: \cite{Co77}.
For compact groups, this problem has been solved by Wassermann \cite{Wa06} on the AFD $\II_1$ factor.
More recently, by using conformal field theory, Bischoff and Karmakar \cite{BK24} have constructed $\mathbb{T}^n$-kernels on the AFD factor of type $\III_1$ realizing a given class in $\mathrm{H}^3(\mathbb{T}^n;\mathbb{T})$.
Marrakchi \cite{Ma25} has proved that if $G$ is a compact connected abelian group, every class in $\mathrm{H}^3(G;\mathbb{T})$ can be realized as an obstruction for a $G$-kernel on a full $\II_1$ factor.
Recall that a factor $M$ is said to be full if the inner automorphism group $\Int(M)$ is a closed subgroup of $\Aut(M)$.
By generalizing Wassermann's construction, we extend these results as follows.

\begin{thm}\label{mainA}
    Let $G$ be a compact group and let $c\colon G^3\to\mathbb{T}$ be a measurable 3-cocycle.
    Then there exist a full $\II_1$ factor $M$ and a $G$-kernel $\kappa\colon G\to\Out(M)$ whose lifting obstruction is $\Ob(\kappa)=[c]$.
\end{thm}

Let $\kappa$ be a $G$-kernel on a full factor $M$.
Assume further that $\kappa$ has a continuous local section to $\Aut(M)$.
Let $\tilde{G}=\varepsilon_M^{-1}(\kappa(G))<\Aut(M)$ denote the inverse image of $\kappa(G)$ under the canonical quotient $\varepsilon_M\colon\Aut(M)\to\Out(M)$.
Then $\tilde{G}\to G$ is a principal $\Int(M)=P\U(M)$-bundle over $G$.
This bundle structure plays a crucial role in the geometric realization of certain topological groups, specifically in the context of string groups.

The string group $\String(n)$ is defined as the 3-connected cover of the spin group $\Spin(n)$.
More generally, $\String(G)$ denotes the 3-connected cover of a compact, simple, and simply-connected Lie group $G$.
Several geometric constructions of $\String(G)$ are known.
For instance, Stolz \cite{St96} has constructed a topological model for $\String(G)$ as an extension of $G$ by the gauge transformation group $\Gau(P)$ where $P$ is a principal $P\U(\mathcal{H})$-bundle.
Nikolaus, Sachse, and Wockel \cite{NSW13} have provided an infinite-dimensional Lie group model for $\String(G)$.
However, an explicit construction of the underlying principal $P\U(\mathcal{H})$-bundle over $G$ has remained elusive.
We resolve this problem by combining the theory of principal bundles with that of $G$-kernels.

\begin{cor}\label{mainB}
    Let $G$ be a compact, simple, and simply-connected Lie group.
    Then there exist a full $\II_1$ factor $M$ and a short exact sequence of Polish groups $P\U(M)\longinto\tilde{G}\longonto G$, which is a Hurewicz fibration.
    The group $\tilde{G}$ is a Polish group model for $\String(G)$.
\end{cor}

Next we treat $G$-kernels with trivial obstruction class, which is identified with cocycle actions.
To handle these actions, particularly for cohomology vanishing problems, we need a stronger form of outerness analogous to the Rohlin property known for discrete amenable group actions and flows.

The classification of group actions is a fundamental subject in the theory of operator algebras.
Connes, Jones, and Ocneanu pioneer the study of asymptotic properties via central sequence algebras.
In the study of flows (i.e., $\mathbb{R}$-actions), a higher degree of ``outerness'' is required for classification.
Consequently, the Rohlin property has been introduced by Kishimoto \cite{Ki96} for \C-algebras and by Kawamuro \cite{Ka00} for the AFD $\II_1$ factor respectively.
Finally, Masuda and Tomatsu \cite{MT16} have classified Rohlin flows on general von Neumann algebras.

On the other hand, actions of compact abelian groups on AFD factors have been studied by Jones, Takesaki, and Kawahigashi \cite{JT84, KT92}.
Later, Masuda and Tomatsu \cite{MT07, MT10} have completed the classification of minimal actions of compact groups on AFD factors.
Their approach relies on dual actions, by exploiting the fact that the dual of a compact group is ``discrete'' in a suitable sense, which enables the application of classification techniques for discrete amenable group actions.
In the context of \C-algebras, Izumi \cite{Iz04a, Iz04b} has introduced the Rohlin property for finite group actions, Hirshberg and Winter \cite{HW07} also has introduced the Rohlin property for compact group actions on unital \C-algebras, and Gardella \cite{Ga17} has generalized it to compact group actions on $\sigma$-unital \C-algebras.
In the second half of this paper, we extend these notions to the von Neumann algebra setting.
Unlike the dual action approach, we focus directly on the action on the equicontinuous part of the central sequence algebra.
Specifically, we introduce the Rohlin property for compact group actions and cocycle actions on von Neumann algebras, and characterize such actions by following the strategy of \cite{MT16}.

\begin{thm}\label{mainC}
    Let $\alpha,\beta\colon G\curvearrowright M$ be Rohlin actions of a compact group $G$ on a factor $M$.
    If $\beta_g\circ\alpha_g^{-1}\in\clInt(M)$ for all $g\in G$, then they are approximately inner conjugate, i.e., there exists $\sigma\in\clInt(M)$ such that $\sigma\circ\alpha\circ\sigma^{-1}=\beta$.
\end{thm}

A cocycle action can be viewed as a $G$-kernel with trivial obstruction class.
For discrete amenable groups, Ocneanu \cite{Oc85} has shown that every cocycle action of such groups on the AFD $\II_1$ factor is perturbed to a genuine action, and Popa \cite{Po21} has extended this result to arbitrary $\II_1$ factor.
In the continuous setting, the problem becomes more subtle.
For locally compact abelian groups, Masuda, Tomatsu \cite{MT16} and Shimada \cite{Sh14} have established that the Rohlin property implies the vanishing of the 2-cohomology.
However, analogous results for non-abelian compact groups have not been known yet.
We prove that strict outerness is sufficient to ensure that a cocycle action on the AFD $\mathrm{II}_1$ factor is perturbed to a genuine action.

\begin{thm}\label{mainD}
    Let $(\alpha,u)$ be a cocycle action of a compact group $G$ on the AFD $\II_1$ factor $R$.
    If $(\alpha,u)$ is strictly outer (i.e., $R'\cap(R\rtimes_{\alpha,u}G)=\mathbb{C}$), then $u$ is a coboundary.
\end{thm}

This paper is organized as follows.
In section 2, we present the proof of Theorem \ref{mainA} and explain string groups to derive Corollary \ref{mainB}.
In section 3, we define the Rohlin property for compact group actions and provide their classification and characterization (Theorem \ref{mainC}).
Finally, section 4 addresses the vanishing of the 1- and 2-cohomology of compact group actions (Theorem \ref{mainD}).

{
\def\addcontentsline#1#2#3{}\relax
\subsection*{Acknowledgement}
This paper is the Master's thesis of the author.
The author is thankful to Narutaka Ozawa and Yusuke Isono for their invaluable comments and support.
The author is also grateful to Yosuke Kubota for his indispensable advice regarding Section \ref{sec:String}, which is largely based on his suggestions, and to Toshihiko Masuda for helpful comments on group actions.
The author utilized Gemini 3.0 Pro to refine the English phrasing.
}

\tableofcontents

\renewcommand{\thethm}{\thesection.\arabic{thm}}
\section{Preliminaries}

\subsection{Notations}

Throughout this paper, all groups are assumed to be second-countable, and all von Neumann algebras are assumed to have separable predual unless it is ultraproduct.

The closed unit ball of a Banach space $E$ is denoted by $(E)_1$.
For a von Neumann algebra $M$, we denote its center, unitary group, projective unitary group, automorphism group, and inner automorphism group by $\Z(M)$, $\U(M)$, $P\U(M)$, $\Aut(M)$, and $\Int(M)$, respectively.
We set $\Out(M)=\Aut(M)/\Int(M)$ and let $\varepsilon_M\colon\Aut(M)\to\Out(M)$ be the canonical quotient map.
The automorphism group $\Aut(M)$ acts isometrically on the predual $M_{\ast}$ canonically:
$\sigma(\varphi)=\varphi\circ\sigma^{-1}$ for $\sigma\in\Aut(M)$ and $\varphi\in M_{\ast}$.
We equip $\Aut(M)$ with the $u$-topology, that is, a net $\alpha_i$ in $\Aut(M)$ converges to $\alpha$ if and only if $\|\alpha_i(\varphi)-\alpha(\varphi)\|\to 0$ for all $\varphi\in M_{\ast}$.
Note that $\Aut(M)$ is a Polish group whenever $M$ has separable predual.
The approximately inner automorphism group $\clInt(M)$ is the closure of $\Int(M)$ in $\Aut(M)$.
We say that a factor $M$ is \emph{full} if $\Int(M)=\clInt(M)$.

For a normal state $\varphi$ on $M$, we consider the two norms on $M$ given by
\[\|x\|_{\varphi}=\varphi(x^{\ast}x)^{1/2},\text{ and }\|x\|_{\varphi}^{\sharp}=\varphi\left(\frac{x^{\ast}x+xx^{\ast}}{2}\right)^{1/2}.\]
If $\tau$ is a distinguished trace on $M$, we write $\|\cdot\|_2$ instead of $\|\cdot\|_{\tau}=\|\cdot\|_{\tau}^{\sharp}$.
We denote the AFD factors of type $\II_1$, $\II_{\infty}$, $\III_{\lambda}\ (0<\lambda<1)$ and $\III_1$ by $R_0$ or $R$, $R_{0,1}$, $R_{\lambda}$ and $R_{\infty}$, respectively.

The left translation action of a group $G$ is denoted by $\mathcal{L}\colon G\curvearrowright L^{\infty}(G)$; $(\mathcal{L}_gf)(h)=f(g^{-1}h)$.
We will write $u\colon G\times G\to H$, $v\colon G\to H$ and $w\in H$ for a 2-cochain, a 1-cochain and an element (0-cochain), respectively.
The evaluation value of a map $\alpha\colon X\to\Aut(M)$ (or $n$-cochain $c$) may be denoted by $\alpha_x$ instead of $\alpha(x)$ (resp.\ $c_{g_1,\ldots,g_n}$).
Given $\alpha\colon X\to\Aut(M)$ and $\sigma\in\Aut(M)$, we write $\sigma\circ\alpha\circ\sigma^{-1}$ for the map $X\ni x\mapsto\sigma\circ\alpha_x\circ\sigma^{-1}\in\Aut(M)$.
More generally, the notation $\pi(\alpha)$ denotes the composition $\pi\circ\alpha$ for a group homomorphism $\pi\colon\Aut(M)\to H$.

Let $\alpha\colon G\curvearrowright M$, $\beta^1\colon H\curvearrowright M$, $\beta^2\colon G\curvearrowright N$ and $\beta^3\colon H\curvearrowright N$ be group actions such that $\alpha$ and $\beta^1$ are commuting.
We define the new actions $\alpha\times\beta^1\colon G\times H\curvearrowright M$, $\alpha\otimes\beta^2\colon G\curvearrowright M\otimes N$, and $\alpha\boxtimes\beta^3\colon G\times H\curvearrowright M\otimes N$ by $(\alpha\times\beta^1)_{g,h}=\alpha_g\circ\beta_h^1$, $(\alpha\otimes\beta^2)_g=\alpha_g\otimes\beta_g^2$ and $(\alpha\boxtimes\beta^3)_{g,h}=\alpha_g\otimes\beta_h^3$.

We will use Moore's measurable group cohomology for locally compact groups.
Our references are \cite{Mo64, Mo76}.
In what follows, measurability refers to the Borel-Haar-measurability, and we always identify functions that differ only on a null set.

\begin{defn}
    Let $G$ be a locally compact group and $A$ a Polish $G$-module via the action $\sigma\colon G\curvearrowright A$.
    The group of \emph{$n$-cochains}, denoted by $\mathrm{C}_{\sigma}^n(G;A)$, consists of all locally bounded measurable functions from $G^n$ to $A$.
    The coboundary operator $\delta^n\colon\mathrm{C}_{\sigma}^n(G;A)\to\mathrm{C}_{\sigma}^{n+1}(G;A)$ is defined by
    \begin{align*}
        (\delta^nc)(g_1,\ldots,g_{n+1}) &= \sigma_{g_1}(c(g_2,\ldots,g_{n+1}))+\sum_{i=1}^n(-1)^ic(g_1,\ldots,g_ig_{i+1},\ldots,g_{n+1})\\
        &\mathrel{\phantom{=}}+(-1)^{n+1}c(g_1,\ldots,g_n).
    \end{align*}
    We define the group of \emph{$n$-cocycles} by $\mathrm{Z}_{\sigma}^n(G;A)=\ker(\delta^n)$, and the group of \emph{$n$-coboundaries} by $\mathrm{B}_{\sigma}^n(G;A)=\im(\delta^{n-1})$.
    The $n$-cohomology group is defined as the quotient $\mathrm{H}_{\sigma}^n(G;A)=\mathrm{Z}_{\sigma}^n(G;A)/\mathrm{B}_{\sigma}^n(G;A)$.
    An $n$-cocycle $c$ is said to be \emph{normalized} if $c(g_1,\ldots,g_n)=0$ whenever $g_i=e$ for some $i$.
    When the action $\sigma$ is trivial, we simply omit the subscript $\sigma$.
\end{defn}

With the notation above, let $I(G)$ denote the $G$-module consisting of all measurable maps from $G$ to $A$, equipped with the action $I(\sigma)\colon G\curvearrowright I(A)$ defined by $[I(\sigma)_gf](h)=\sigma_g(f(g^{-1}h))$.

\begin{lem}[Shapiro lemma]
    For any locally compact group $G$ and any $G$-module $A$, the cohomology group $\mathrm{H}_{I(\sigma)}^n(G;I(A))$ is trivial.
\end{lem}
\begin{proof}
    This follows from \cite[Theorem 2]{Mo76}.
\end{proof}

Consider a cocycle $c\in\mathrm{Z}_{\sigma}^n(G;A)$.
Suppose that $c$ takes values in a submodule $B\subset A$ and that $I(B)$ admits an equivariant embedding into $A$.
Then $c$ is a coboundary.
This observation will be used throughout this paper.

\begin{defn}[cocycle action]
    Let $G$ be a locally compact group and $M$ be a von Neumann algebra.
    A pair of measurable maps $\alpha\colon G\to\Aut(M)$ and $u\colon G\times G\to\U(M)$ is said to be a \emph{cocycle action} if it satisfies
    \begin{enumerate}
        \item $\alpha_g\circ\alpha_h=\Ad(u_{g,h})\circ\alpha_{gh}$, and
        \item $\alpha_g(u_{h,k})u_{g,hk}=u_{g,h}u_{gh,k}$
    \end{enumerate}
    for (almost) every $g,h,k\in G$.
    If $u\equiv1$, then $\alpha$ is an \emph{action}.
\end{defn}

We assume in addition that cocycles are normalized:
$u_{g,h}=1$ if $g$ or $h$ is $e$.
The condition (1) means that $\alpha$ induces a group homomorphism $\varepsilon_M\circ\alpha\colon G\to\Out(M)$.
Moreover, the condition (2) leads to the existence of a \emph{covariant representation} as follows:
Let $\pi_{\alpha}\colon M\to M\otimes L^{\infty}(G)\subset M\otimes\mathbb{B}(L^2(G))$ and $\lambda^u\colon G\to\U(M\otimes\mathbb{B}(L^2(G)))$ be maps given by
\[(\pi_{\alpha}(x)\xi)(g)=\alpha_{g^{-1}}(x)\xi(g),\quad(\lambda_h^u\xi)(g)=u_{g^{-1},h}\xi(h^{-1}g)\]
for $g,h\in G$ and $\xi\in \mathcal{H}\otimes L^2(G)$, where $\mathcal{H}$ is the Hilbert space on which $M$ acts.
Then it follows that $\lambda_h^u\pi_{\alpha}(x)=\pi_{\alpha}(\alpha_h(x))\lambda_h^u$ and $\lambda_g^u\lambda_h^u=u_{g,h}\lambda_{gh}^u$.
The \emph{cocycle crossed product} is a von Neumann algebra $\pi_{\alpha}(M)\vee\{\lambda^u\}''$ and denoted by $M\rtimes_{\alpha,u}G$.
We may omit the notation $\pi_{\alpha}$ and identify $M\subset M\rtimes_{\alpha,u}G$.
We say that a cocycle action $(\alpha,u)$ is \emph{strictly outer} if $M'\cap (M\rtimes_{\alpha,u}G)=\Z(M)$.

Let $v\colon G\to\U(M)$ be an arbitrary measurable map.
Then $v$ \emph{perturbs} the cocycle action $(\alpha,u)$ to $(\alpha',u')$, where 
\[\alpha'_g=\Ad(v_g)\circ\alpha_g,\quad u'_{g,h}=v_g\alpha_g(v_h)u_{g,h}v_{gh}^{\ast}.\]
A new cocycle action $(\alpha',u')$ is said to be a \emph{unitary perturbation of $(\alpha,u)$}.
In this case, there exists an isomorphism $M\rtimes_{\alpha,u}G\cong M\rtimes_{\alpha',u'}G$ determined by $x\mapsto x$ for all $x\in M$ and $\lambda_g^u\mapsto v_g^{\ast}\lambda_g^{u'}$.
Indeed, consider the unitary operator $V$ on $\mathcal{H}\otimes L^2(G)$ such that
\[(V\xi)(g)=v_g\xi(g).\]
One can check that $V(\pi_{\alpha}(x))=\pi_{\alpha'}(x)V$ and $V\lambda_g^u=v_g^{\ast}\lambda_g^{u'}V$, and hence $\Ad(V)$ gives the desired isomorphism.
If a cocycle action $(\alpha,u)$ is perturbed to an action $\alpha'=(\alpha',1)$, then $u$ is called a \emph{coboundary}.
Explicitly, this means that there exists $v\colon G\to\U(M)$ such that $u_{g,h}=\alpha_g(v_h^{\ast})v_g^{\ast}v_{gh}$.
If both $(\alpha,u)$ and $(\alpha',u')$ are actions (i.e., $u\equiv u'\equiv 1$), then the map $v$ is called an \emph{$\alpha$-cocycle}.
In other words, an $\alpha$-cocycle is a measurable map $v\colon G\to\U(M)$ satisfying the cocycle identity: $v_g\alpha_g(v_h)=v_{gh}$.
An $\alpha$-cocycle $v$ is called a \emph{coboundary} if there exists a unitary $w\in\U(M)$ such that $v_g=w\alpha_g(w^{\ast})$.
Two actions $\alpha$ and $\alpha'$ are said to be \emph{conjugate} if there exists an automorphism $\sigma\in\Aut(M)$ such that $\sigma\circ\alpha\circ\sigma^{-1}=\alpha'$.
Note that strict outerness of a (cocycle) action is preserved under conjugation and perturbation.

Recall that the \emph{continuous core} of a von Neumann algebra $M$ is the crossed product $\tilde{M}=M\rtimes_{\sigma^{\varphi}}\mathbb{R}$ associated with the modular action of a faithful normal semifinite weight $\varphi$.
We denote by $\lambda_t^{\varphi}\in\tilde{M}$ the unitary implementing $\sigma_t^{\varphi}$.
The continuous core $\tilde{M}$ admits a trace and a trace-scaling flow $\theta=\hat{\sigma^{\varphi}}$.
The \emph{canonical extension} of $\alpha\in\Aut(M)$ is the automorphism $\tilde{\alpha}\in\Aut(\tilde{M})$ defined by
\[\tilde{\alpha}(x)=\alpha(x),\quad \tilde{\alpha}(\lambda_t^{\varphi})=[D\varphi\circ\alpha^{-1}:D\varphi]_t\lambda_t^{\varphi}\]
for $x\in M$ and $t\in\mathbb{R}$.
The \emph{Connes--Takesaki module} $\mod(\alpha)$ is defined by the restriction $\mod(\alpha)=\tilde{\alpha}|_{\Z(\tilde{M})}\in\Aut(\Z(\tilde{M}))$.
The mappings $\alpha\mapsto\tilde{\alpha}$ and $\alpha\mapsto\mod(\alpha)$ are continuous.
See \cite{FT01} for more details.

\begin{defn}[$G$-kernel]
    A \emph{$G$-kernel} on a von Neumann algebra $M$ is a group homomorphism $\kappa\colon G\to\Out(M)$ that admits a measurable lift $\alpha\colon G\to\Aut(M)$.
\end{defn}

We always assume that $\kappa$ is faithful.
Note that if $M$ is full, the existence of a measurable lift $\alpha$ of $\kappa$ is equivalent to the continuity of $\kappa$, because any measurable homomorphism into a Polish group is automatically continuous.
Since $\alpha_g\circ\alpha_h\circ\alpha_{gh}^{-1}$ belongs to $\Int(M)$, there exists a measurable map $u\colon G\times G\to\U(M)$ such that $\alpha_g\circ\alpha_h=\Ad(u_{g,h})\circ\alpha_{gh}$.
The \emph{lifting obstruction} (or simply \emph{obstruction}) of $\kappa$ is the 3-cohomology class $\Ob(\kappa)\in\mathrm{H}^3(G;\mathbb{T})$ given by $\Ob(\kappa)=[c]$,
\[c(g,h,k)u_{g,h}u_{gh,k}=\alpha_g(u_{h,k})u_{g,hk}.\]
The class $[c]$ is trivial if $\kappa$ is \emph{split}, namely, we can choose $\alpha$ as a homomorphism.
The converse also holds if $M$ is properly infinite (\cite{Su80}).

\subsection{Minimal actions of compact groups}

The uniqueness up to conjugacy of minimal actions for compact groups (and more generally, for compact Kac algebras with amenable dual) on AFD factors of type $\II$ has been proved in \cite{MT07, MT10} (see also \cite{Oc85}).
An action $\alpha\colon G\curvearrowright M$ on a factor is called \emph{minimal} if it is faithful and $(M^{\alpha})'\cap M=\mathbb{C}$, and the action $\alpha$ is called \emph{strictly outer} if $M'\cap (M\rtimes_{\alpha}G)=\mathbb{C}$.
For compact group actions, it is known that minimality is equivalent to strict outerness (\cite{Va01}).

\begin{exam}[Wassermann]
    Let $G$ be a compact group and $v\colon G\to P\U(\mathbb{M}_n)$ a faithful projective representation.
    Then the infinite tensor product $\gamma=\bigotimes\Ad{v}\colon G\curvearrowright\bigotimes\mathbb{M}_n\cong R$ is a minimal action.
    All minimal actions on AFD $\II_1$ factor are conjugate to $\gamma$.
\end{exam}

\begin{thm}[\cite{MT07, MT10}]\label{thm:uniqmin}
    Let $M$ be an AFD factor of type $\II$, $G$ a compact group, and $\gamma$ a minimal action of $G$ on $R$.
    Then every minimal action of $G$ on $M$ is conjugate to $\id\otimes\gamma$.
\end{thm}

% From this theorem, we observe that any minimal action of a non-trivial closed subgroup $H<G$ on $R$ can be extended to a minimal action of the whole group $G$.

Recall that the \emph{dual} of a compact group $G$ is the set of all equivalence classes of irreducible unitary representations of $G$, and is denoted by $\hat{G}$.
The \emph{dual action} of a cocycle action $(\alpha,u)\colon G\curvearrowright M$ is the family of *-homomorphisms $(\hat{\alpha}_{\pi}\colon M\rtimes_{\alpha,u}G\to(M\rtimes_{\alpha,u}G)\otimes\mathbb{B}(\mathcal{H}_{\pi}))_{\pi\in\hat{G}}$ determined by
\[\hat{\alpha}_{\pi}(x)=x\otimes 1,\quad \hat{\alpha}_{\pi}(\lambda_g^u)=\lambda_g^u\otimes\pi_g\]
for $x\in M$ and $g\in G$.
For more details, see \cite{MT07, MT10}.

\subsection{Ultraproduct von Neumann algebras}

Standard references are \cite[Chapter 5]{Oc85}, \cite{AH14} and \cite{MT16}.
Let $M$ be a von Neumann algebra.
We fix a free ultrafilter $\omega$ on $\mathbb{N}$.
Let us consider the following \C-subalgebras of $\ell^{\infty}(\mathbb{N},M)$:
\begin{align*}
    \mathscr{T}_{\omega}(M) &= \{(a^{\nu})_{\nu}:\text{$a^{\nu}\to0$ $\ast$-strongly as $\nu\to\omega$}\},\\
    \mathscr{C}_{\omega}(M) &= \{(a^{\nu})_{\nu}:\text{$\|[a^{\nu},\varphi]\|\to0$ as $\nu\to\omega$ for all $\varphi\in M_{\ast}$}\},\\
    \mathscr{N}_{\omega}(M) &= \{(a^{\nu})_{\nu}:(a^{\nu})_{\nu}\text{ normalizes $\mathscr{T}_{\omega}(M)$ in $\ell^{\infty}(\mathbb{N},M)$}\}.
\end{align*}
Then $\mathscr{T}_{\omega}(M)\subset\mathscr{C}_{\omega}(M)\subset\mathscr{N}_{\omega}(M)$.
Note that if $M$ is tracial then $\mathscr{N}_{\omega}(M)=\ell^{\infty}(\mathbb{N},M)$.
We call the quotient \C-algebras $M^{\omega}=\mathscr{N}_{\omega}(M)/\mathscr{T}_{\omega}(M)$ (resp.\ $M_{\omega}=\mathscr{C}_{\omega}(M)/\mathscr{T}_{\omega}(M)$) \emph{ultraproduct algebra} (resp.\ \emph{central sequence algebra}).
They are in fact von Neumann algebras.
We regard $M$ and $M_{\omega}$ as von Neumann subalgebras of $M^{\omega}$.
Under this identification, we have the inclusion $M_{\omega}\subset M'\cap M^{\omega}$.

We define a faithful normal conditional expectation $E$ from $M^{\omega}$ onto $M$ by $E((a^{\nu})_{\nu})=\wlim_{\nu\to\omega}a^{\nu}$.
The restriction $\tau_{\omega}\coloneqq E|_{M_{\omega}}$ yields a faithful normal trace with values in $\Z(M)$.
We often use the notation $\varphi^{\omega}=\varphi\circ E$ for $\varphi\in M_{\ast}$ and the 2-norm on $M_{\omega}$ with respect to the trace $\tau_{\omega}$.
Every $\sigma\in\Aut(M)$ induces $\sigma^{\omega}\in\Aut(M^{\omega})$ given by $\sigma^{\omega}((a^{\nu})_{\nu})=(\sigma(a^{\nu}))_{\nu}$.
We also write $\sigma$ for $\sigma^{\omega}$ when no confusion arises.

\begin{lem}[cf.\ {\cite[Lemma 2.8]{MT16}}]
    Let $Q$ be a separable factor of type $\I$.
    Then canonically $(M\otimes Q)^{\omega}\cong M^{\omega}\otimes Q$ and $(M\otimes Q)_{\omega}\cong M_{\omega}$.
\end{lem}

\begin{lem}[cf.\ {\cite[Lemma 4.36]{AH14}}]\label{lem:normoffunctional}
    Let $(x^{\nu})_{\nu},(y^{\nu})_{\nu}\in M^{\omega}$ and $\varphi,\psi\in M_{\ast}$.
    Then
    \[\lim_{\nu\to\omega}\|x^{\nu}\varphi-\psi y^{\nu}\|_{M_{\ast}}=\|(x^{\nu})_{\nu}\varphi^{\omega}-\psi^{\omega}(y^{\nu})_{\nu}\|_{(M^{\omega})_{\ast}}.\]
\end{lem}
\begin{proof}
    The inequality $\lim_{\nu\to\omega}\|x^{\nu}\varphi-\psi y^{\nu}\|\ge\|(x^{\nu})_{\nu}\varphi^{\omega}-\psi^{\omega}(y^{\nu})_{\nu}\|$ is clear.
    For the converse, choose $a^{\nu}\in M$ such that $\|a^{\nu}\|=1$ and $\inn{a^{\nu},x^{\nu}\varphi-\psi y^{\nu}}=\|x^{\nu}\varphi-\psi y^{\nu}\|$.
    By \cite[Corollary 3.16]{AH14}, there exist $b^{\nu},c^{\nu}$ and $d^{\nu}\in M$ satisfying
    \begin{itemize}
        \item $a^{\nu}=b^{\nu}+c^{\nu}+d^{\nu}$;
        \item $(b^{\nu})_{\nu}$ belongs to $\mathscr{N}_{\omega}(M)$;
        \item $\|(b^{\nu})_{\nu}\|\le 1$;
        \item $c^{\nu}\to0$, $(d^{\nu})^{\ast}\to 0$ strongly as $\nu\to\omega$.
    \end{itemize}
    According to \cite[Corollary 3.11]{AH14}, the sequences $(x^{\nu}c^{\nu})_{\nu},(c^{\nu}x^{\nu})_{\nu},(y^{\nu}d^{\nu\ast})_{\nu}$ and $(d^{\nu\ast}y^{\nu})_{\nu}$ converge to 0 strongly (and hence weakly) as $\nu\to\omega$.
    Thus, we have
    \begin{align*}
        \lim_{\nu\to\omega}\|x^{\nu}\varphi-\psi y^{\nu}\| &= \lim_{\nu\to\omega}\inn{a^{\nu},x^{\nu}\varphi-\psi y^{\nu}}\\
        &= \lim_{\nu\to\omega}\inn{b^{\nu},x^{\nu}\varphi-\psi y^{\nu}}+\inn{(c^{\nu}+d^{\nu})x^{\nu},\varphi}+\inn{y^{\nu}(c^{\nu}+d^{\nu}),\psi}\\
        &= \lim_{\nu\to\omega}\inn{b^{\nu},x^{\nu}\varphi-\psi y^{\nu}}\\
        &= \inn{(b^{\nu})_{\nu},(x^{\nu})_{\nu}\varphi^{\omega}-\psi^{\omega}(y^{\nu})_{\nu}}.
    \end{align*}
    This value is bounded by $\|(x^{\nu})_{\nu}\varphi^{\omega}-\psi^{\omega}(y^{\nu})_{\nu}\|$.
\end{proof}

Recall that every second countable compact group has an invariant metric.

\begin{defn}[equicontinuous parts, cf.\ {\cite[Definition 3.1, 3.4 and 3.9]{MT16}}]
    Let $X$ be a compact space with a probability measure $\mu$, $M$ a von Neumann algebra with a faithful normal state $\varphi$ and $\omega$ a free ultrafilter.
    \begin{enumerate}
        \item A sequence of maps $(f^{\nu}\colon X\to M)_{\nu\in\mathbb{N}}$ is said to be \emph{$\omega$-equicontinuous} if for each $x\in X$ and $\varepsilon>0$, there exists a neighborhood $U$ of $x\in X$ such that
        \[\{\nu\in\mathbb{N}:\text{$\|f^{\nu}(x)-f^{\nu}(y)\|_{\varphi}^{\sharp}<\varepsilon$ whenever $y\in U$}\}\in\omega.\]
        \item Let $\alpha\colon X\to\Aut(M)$ be a measurable map.
        A sequence $(a^{\nu})_{\nu}\in\ell^{\infty}(\mathbb{N},M)$ is said to be \emph{$(\alpha,\omega)$-equicontinuous} if for each $\varepsilon>0$, there exists a compact subset $K\subset X$ such that
        \begin{itemize}
            \item $\alpha|_K$ is continuous;
            \item $\mu(X\setminus K)<\varepsilon$;
            \item the sequence $(K\ni x\mapsto\alpha_x(a^{\nu})\in M)_{\nu}$ is $\omega$-equicontinuous.
        \end{itemize}
        Let $\mathscr{E}_{\alpha}^{\omega}(M)$ denote the set of all $(\alpha,\omega)$-equicontinuous sequences in $\ell^{\infty}(\mathbb{N},M)$, and then $\mathscr{E}_{\alpha}^{\omega}(M)$ is a \C-subalgebra containing $\mathscr{T}_{\omega}(M)$.
        We write $M_{\alpha}^{\omega}\coloneqq (\mathscr{E}_{\alpha}^{\omega}(M)\cap\mathscr{N}_{\omega}(M))/\mathscr{T}_{\omega}(M)$ and $M_{\omega,\alpha}\coloneqq(\mathscr{E}_{\alpha}^{\omega}(M)\cap\mathscr{C}_{\omega}(M))/\mathscr{T}_{\omega}(M)= M_{\alpha}^{\omega}\cap M_{\omega}$.
    \end{enumerate}
\end{defn}

The above definition does not depend on the choice of $\varphi$.
The \C-subalgebras $M_{\alpha}^{\omega}$ and $M_{\omega,\alpha}$ are von Neumann subalgebras (\cite[Lemma 3.10]{MT16}).
By Lusin's theorem, we have $M\subset M_{\alpha}^{\omega}$.

Let $G$ be a compact group and $(\alpha,u)$ a cocycle action of $G$ on a von Neumann algebra $M$.
Then $(\alpha,u)$ is a cocycle action on $M_{\alpha}^{\omega}$ and $\alpha$ is an action on $M_{\omega,\alpha}$ (\cite[Lemma 3.13]{MT16}).

\begin{lem}[cf.\ {\cite[Lemma 3.3]{MT16}}]\label{lem:limit}
    Let $X$ be a compact space and $f^{\nu}\colon X\to \mathbb{C}$ be an $\omega$-equicontinuous sequence of uniformly bounded measurable functions.
    Then the convergence $\lim_{\nu\to\omega}f^{\nu}(x)$ is uniform on $x\in X$.
    In particular, if $X$ admits a probability measure, then
    \[
    \lim_{\nu\to\omega}\int_Xf^{\nu}(x)\,dx=\int_X\lim_{\nu\to\omega}f^{\nu}(x)\,dx.
    \]
\end{lem}

\begin{lem}[cf.\ {\cite[Lemma 3.21]{MT16}}]\label{lem:Borellift}
    Let $X$ be a compact metric space and $V\colon X\to\U(M^{\omega})$ be a continuous map from $X$ to the unitary group of the ultrapower of some separable von Neumann algebra $M$.
    Then there exists an equicontinuous sequence of Borel maps $v^\nu\colon X\to\U(M)$ such that $(v_x^{\nu})_{\nu}=V_x$ for all $x\in X$.
\end{lem}
\begin{proof}
    Fix a faithful normal state $\varphi$ on $M$.
    First we take a sequence of finite Borel partitions $X=\bigsqcup_{i=1}^{N_k}X_i^{(k)}$ and points $x_i^{(k)}\in X_i^{(k)}$ such that
    \begin{enumerate}[label=(\alph*)]
        \item $\diam(X_i^{(k)})<1/k$;
        \item $x_i^{(k)}=x_j^{(k+1)}$ if $x_i^{(k)}\in X_j^{(k+1)}$.
    \end{enumerate}
    Fix a representing sequence $(v_{x_i^{(k)}}^{\nu})_{\nu}$ consisting of unitaries for $V_{x_i^{(k)}}$
    We find $F_k\in\omega$ such that
    \begin{enumerate}[label=(\alph*)]\setcounter{enumi}{2}%
        \item $F_k\supset F_{k+1}$;
        \item $\min{F_k}\ge k$;
        \item $\nu\in F_k$ implies $\|v_{x_i^{(k)}}^{\nu}-v_{x_j^{(k)}}^{\nu}\|_{\varphi}^{\sharp}<\|V_{x_i^{(k)}}-V_{x_j^{(k)}}\|_{\varphi^{\omega}}^{\sharp}+1/k$.
    \end{enumerate}
    For $\nu\in F_k\setminus F_{k+1}$ and $x\in X_i^{(k)}$, we define $v_x^{\nu}\coloneqq v_{x_i^{(k)}}^{\nu}$.
    Then the map $X\ni x\mapsto v_x^{\nu}\in\mathcal{U}(M)$ is Borel.
    It is clear by definition that $(v_x^{\nu})_{\nu}=V_x$ whenever $x=x_i^{(k)}$ for some $k$ and $i$.
    For any general points $x,y\in X$, we can take sequences $(i_k)_k,(j_k)_k$ such that $x\in X_{i_k}^{(k)}$ and $y\in X_{j_k}^{(k)}$ for each $k$.
    We see that $\lim_{k\to\infty} x_{i_k}^{(k)}=x$, $\lim_{k\to\infty}x_{j_k}^{(k)}=y$ in $X$ and
    \begin{align*}
        \|(v_x^{\nu})_{\nu}-(v_y^{\nu})_{\nu}\|_{\varphi^{\omega}}^{\sharp} &= \lim_{\nu\to\omega}\|v_x^{\nu}-v_y^{\nu}\|_{\varphi}^{\sharp}\\
        &\le \liminf_{k\to\infty} (\|V_{x_{i_k}^{(k)}}-V_{x_{j_k}^{(k)}}\|_{\varphi^{\omega}}^{\sharp}+1/k)\\
        &= \|V_x-V_y\|_{\varphi^{\omega}}^{\sharp},
    \end{align*}
    where $\varphi$ is a faithful normal state on $M$.
    Therefore $X\ni x\mapsto (v_x^{\nu})_{\nu}\in\mathcal{U}(M^{\omega})$ is continuous.
    By the density of $(x_i^{(k)})_{k,i}\subset X$, it follows that $(v_x^{\nu})_{\nu}=V_x$ for all $x\in X$.
\end{proof}

Suppose now that $M$ is a factor and let $L^{\infty}(X)\subset M_{\omega}$ be an abelian subalgebra.
We consider the probability measure $\mu=\tau_{\omega}|_{L^{\infty}(X)}$ on $X$.
Then there exists an isomorphism $\Theta\colon M\otimes L^{\infty}(X)\cong M\vee L^{\infty}(X)\subset M^{\omega}$ given by $\Theta(a\otimes f)=af$.
Indeed, $E\circ\Theta(a\otimes f)=E(af)=aE(f)=a\mu(f)$ hence $\Theta$ is isometric with respect to the 2-norm.

\begin{lem}\label{lem:AdW}
    In the above setting, let $W=(w^{\nu})_{\nu}\in\U(M\otimes L^{\infty}(X))\subset\U(M^{\omega})$ be a unitary.
    Then for any $\varphi,\psi\in M_{\ast}$,
    \[\lim_{\nu\to\omega}\|\Ad{w^{\nu}}(\varphi)-\psi\|=\|\Ad{W}(\varphi^{\omega})-\psi^{\omega}\|\le\int_X\|\Ad[W(x)](\varphi)-\psi\|\,dx.\]
\end{lem}
\begin{proof}
    The first equality follows from Lemma \ref{lem:normoffunctional}.
    First we suppose that $W$ is a simple function, namely, $W$ has the form $\sum_iw_i\otimes\chi_{X_i}$ for some finite partition $X=\bigsqcup_i X_i$ and $w_i\in\U(M)$.
    Let $\phi_i=\Ad{w_i}(\varphi)-\psi$.
    Since the projections $\chi_{X_i}$ are mutually orthogonal and commute with $w_i$ and $\varphi^{\omega}$, it follows that
    \begin{align*}
        \|\Ad{W}(\varphi^{\omega})-\psi^{\omega}\| &= \|\sum_i(\Ad{w_i}(\varphi^{\omega})-\psi^{\omega})\chi_{X_i}\|\\
        &= \|\sum_i\phi_i^{\omega}\chi_{X_i}\|\\
        &\le \sum_i\|\phi_i^{\omega}\chi_{X_i}\|.
    \end{align*}
    Here, for any $\phi\in M_{\ast}$ and a projection $p\in M_{\omega}$, $\|\phi^{\omega}p\|=\|\phi\|\tau_{\omega}(p)$.
    Indeed, let $\phi=u|\phi|$ be a polar decomposition
    Then $\|\phi^{\omega}p\|=\|u|\phi|^{\omega}p\|\le\||\phi|^{\omega}p\|=|\phi|^{\omega}(p)=\|\phi\|\tau_{\omega}(p)$.
    Hence we have
    \[\|\Ad{W}(\varphi^{\omega})-\psi^{\omega}\|\le\sum_i\|\phi_i\|\mu(X_i)=\int_X\|\Ad[W(x)](\varphi)-\psi\|\,dx.\]

    For the general case, any unitary $W\in \U(M\otimes L^{\infty}(X))$ is approximated in SOT by a unitary simple functions.
    Since both sides of the second inequality are continuous for $W$, the conclusion holds.
\end{proof}

\section{\texorpdfstring{$G$}{G}-kernels}

\subsection{Obstruction realization}

Wassermann has shown that every 3-cohomology class of a compact group can be realized as the lifting obstruction of a kernel on the AFD $\II_1$ factor (\cite{Wa06}).
Therefore, the obstruction realization problem has been solved for every McDuff-type factor.

Wassermann's construction is as follows.
Given a 3-cocycle $c\colon G^3\to\mathbb{T}$, consider the function $u\colon G\times G\to \U(L^{\infty}(G))$ defined by $u_{g,h}(g')=c(g',g,h)$, where $\U(L^{\infty}(G))$ is equipped with the right translation $G$-action $\rho$.
Then $u$ satisfies the 2-cocycle identity modulo $\mathbb{T}$, that is,
\[[\rho_g(u_{h,k})u_{g,hk}u_{gh,k}^{-1}u_{g,h}^{-1}](g')=c(g'g,h,k)c(g',g,hk)c(g',gh,k)^{-1}c(g',g,h)^{-1}=c(g,h,k)\]
is a constant function.
Hence $u\otimes u^{\ast}\colon G\times G\to \U(L^{\infty}(G\times G))$ is a 2-cocycle.
Moreover, it is a 2-coboundary.
Indeed, setting $v\colon G\to \U(L^{\infty}(G\times G))$ as $v_g(g_0,g_1)=c(g_0g_1^{-1},g_1,g)$, one can verify that $v_g\rho_g(v_h)v_{gh}^{-1}=u_{g,h}\otimes u_{g,h}^{\ast}$.
Next, consider the inclusions $L^{\infty}(G\times G)\subset L^{\infty}(G)^{\otimes\mathbb{Z}}\subset L^{\infty}(G)^{\otimes\mathbb{Z}}\rtimes\mathbb{Z}$
Let $U$ denote the unitary implementing $1\in\mathbb{Z}$.
There exists a measurable map $\alpha\colon G\to\Aut(L^{\infty}(G)^{\otimes\mathbb{Z}}\rtimes\mathbb{Z})$ defined by
\[\alpha_g(x)=\rho_g(x),\quad \alpha_g(U)=v_gU\]
for $x\in L^{\infty}(G)^{\otimes\mathbb{Z}}$.
It follows that $\alpha_g\circ\alpha_h(U)=\rho_g(v_h)v_gU=v_{gh}(u_{g,h}\otimes u_{g,h}^{\ast})U=v_{g,h}u_{g,h}Uu_{g,h}^{\ast}$ since conjugation by $U$ shifts the tensor indices.
Thus, we have $\alpha_g\circ\alpha_h=\Ad(u_{g,h})\circ\alpha_{gh}$.
The $G$-kernel $\kappa=\varepsilon_M\circ\alpha$ has the obstruction $\Ob(\kappa)=[c]$.

We now generalize the above argument.
Let $n\in\{1,2,\ldots,\infty\}$ and let $\mathbb{F}_n=\inn{s_1,s_2,\ldots,s_n}$ denote the free group of rank $n$.

\begin{lem}\label{lem:freecocycle}
    Let $H$ be a group and $\sigma\colon\mathbb{F}_n\curvearrowright H$ an action of the free group of rank $n$.
    For any $n$-tuple $(g_i)_{i=1}^n\in H^n$, there exists a unique $\sigma$-cocycle $v\colon\mathbb{F}_n\to H$ such that $v_{s_i}=g_i$ for $i\in\{1,2,\ldots,n\}$.
\end{lem}
\begin{proof}
    Let $\tilde{H}=H\rtimes_{\sigma}\mathbb{F}_n$ be the semidirect product group and $\pi\colon\tilde{H}\to\mathbb{F}_n$ denote the canonical surjection.
    By freeness, we find a unique homomorphism $\tilde{v}\colon\mathbb{F}_n\to\tilde{H}$ such that $\tilde{v}_{s_i}=(g_i,s_i)\in\tilde{H}$.
    Since $\pi\circ\tilde{v}$ is the identity map on $\mathbb{F}_n$, there exists a map $v\colon\mathbb{F}_n\to H$ such that $\tilde{v}_s=(v_s,s)$ for all $s\in\mathbb{F}_n$.
    Then $v$ satisfies the cocycle identity $v_s\sigma_s(v_t)=v_{st}$.

    Uniqueness follows from the fact that any $\sigma$-cocycle $v$ satisfying the condition gives rise to such a homomorphism $\tilde{v}$, which is uniquely determined by its values on the generators.
\end{proof}

Let $\sigma\colon\Gamma\curvearrowright P$ be an action of a discrete group $\Gamma$ on a von Neumann algebra $P$.
Then any unitary $\sigma$-cocycle $v\colon\Gamma\to\U(\Z(P))$ induces the automorphism $\hat{\sigma}_v\in\Aut(P\rtimes_{\sigma}\Gamma)$ such that
\[\hat{\sigma}_v(x)=x,\quad \hat{\sigma}_v(\lambda_s)=v_s\lambda_s\]
for $x\in P$ and $s\in\Gamma$.
This correspondence $\mathrm{Z}_{\sigma}^1(\Gamma;\U(\Z(P)))\ni v\mapsto\hat{\sigma}_v\in\Aut(P\rtimes_{\sigma}\Gamma)$ is a group homomorphism.
Let $\Aut_{\sigma}(P)$ denote the group consisting of all $\alpha\in\Aut(P)$ that commute with $\sigma$.
Any $\alpha\in\Aut_{\sigma}(P)$ canonically extend to an automorphism of $P\rtimes_{\sigma}\Gamma$ (still denoted by $\alpha$) by setting $\alpha(\lambda_s)=\lambda_s$.
Then there exists a relation $\alpha\circ\hat{\sigma}_v=\hat{\sigma}_{\alpha(v)}\circ\alpha$, where $\alpha(v)_s=\alpha(v_s)$.
Thus we obtain
\[\mathrm{Z}_{\sigma}^1(\Gamma;\U(\Z(P)))\rtimes\Aut_{\sigma}(P)\subset\Aut(P\rtimes_{\sigma}\Gamma).\]

\begin{thm}\label{thm:main1}
    Let $G$ be a compact group, $c\colon G^3\to\mathbb{T}$ be a measurable 3-cocycle and $n\in\{1,2,\ldots,\infty\}$.
    Then there exists a $G$-kernel $\kappa\colon G\to\Out(M)$ on a factor $M=L^{\infty}(G)^{\otimes\mathbb{F}_n}\rtimes\mathbb{F}_n$ whose lifting obstruction is $\Ob(\kappa)=[c]$.
\end{thm}
\begin{proof}
    Define a diagonal right translation action $\rho\colon G\curvearrowright G^{\mathbb{F}_n}$ and an index shift action (Bernoulli action) $\sigma\colon\mathbb{F}_n\curvearrowright G^{\mathbb{F}_n}$ by
    \[\rho_g((g_s)_s)=(g_sg^{-1})_s,\quad\sigma_t((g_s)_s)=(g_{t^{-1}s})_s\]
    for $g\in G$, $t\in\mathbb{F}_n$ and $(g_s)_{s\in\mathbb{F}_n}\in G^{\mathbb{F}_n}$.
    Note that these two actions commute.
    Consider a map $u\colon G\times G\to L^{\infty}(G^{\mathbb{F}_n})=L^{\infty}(G)^{\otimes\mathbb{F}_n}$ given by
    \[u_{g,h}((g_s)_s)=c(g_e,g,h).\]
    We see that $\rho_g(u_{h,k})u_{g,hk}=c(g,h,k)u_{g,h}u_{gh,k}$, where $\rho\colon G\curvearrowright L^{\infty}(G)^{\otimes\mathbb{F}_n}$ is given by $\rho_g(x)=x\circ\rho_g^{-1}$.
    Indeed, the 3-cocycle identity implies
    \begin{align*}
        [\rho_g(u_{h,k})u_{g,hk}]((g_s)_s) &= u_{h,k}((g_sg)_s)u_{g,hk}((g_s)_s)\\
        &= c(g_eg,h,k)c(g_e,g,hk)\\
        &= c(g,h,k)c(g_e,g,h)c(g_e,gh,k)\\
        &= c(g,h,k)[u_{g,h}u_{gh,k}]((g_s)_s).
    \end{align*}
    Thus the proof is finished once we find $\alpha\colon G\to\Aut(M)$ such that $\alpha_g|_{L^{\infty}(G)^{\otimes\mathbb{F}_n}}=\rho_g$ and $\alpha_g\circ\alpha_h=\Ad(u_{g,h})\circ\alpha_{gh}$.

    By Lemma \ref{lem:freecocycle}, for any $g\in G$, there exists a unitary $\sigma$-cocycle $v_g\colon\mathbb{F}_n\to L^{\infty}(G)^{\otimes\mathbb{F}_n}$ such that
    \[v_{g,s_i}((g_s)_s)=c(g_eg_{s_i}^{-1},g_{s_i},g)\]
    holds for all $i\in\{1,2,\ldots,n\}$.
    Then there exists a unique automorphism $\alpha_g=\hat{\sigma}_{v_g}\circ\rho_g$ on $M=L^{\infty}(G)^{\otimes\mathbb{F}_n}\rtimes_{\sigma}\mathbb{F}_n$, which satisfies
    \[\alpha_g(x)=\rho_g(x),\quad\alpha_g(\lambda_s)=v_{g,s}\lambda_s\]
    for all $x\in L^{\infty}(G)^{\otimes\mathbb{F}_n}$ and $s\in\mathbb{F}_n$.

    Now we note that $\rho_g(v_{h,s_i})v_{g,s_i}=u_{g,h}v_{gh,s_i}\sigma_{s_i}(u_{g,h}^{\ast})$ holds for $g,h\in G$ and $i\in\{1,2,\ldots,n\}$.
    This equation follows from the following calculation:
    \begin{align*}
        [\rho_g(v_{h,s_i})v_{g,s_i}]((g_s)_s) &= v_{h,s_i}((g_sg)_s)v_{g,s_i}((g_s)_s)\\
        &= c(g_eg_{s_i}^{-1},g_{s_i}g,h)c(g_eg_{s_i}^{-1},g_{s_i},g)\\
        &= c(g_e,g,h)c(g_eg_{s_i}^{-1},g_{s_i},gh)\overline{c(g_{s_i},g,h)}\\
        &= u_{g,h}((g_s)_s)v_{gh,s_i}((g_s)_s)\overline{u_{g,h}((g_{s_is})_s)}\\
        &= [u_{g,h}v_{gh,s_i}\sigma_{s_i}(u_{g,h}^{\ast})]((g_s)_s).
    \end{align*}
    Hence we have $\alpha_g\circ\alpha_h(\lambda_{s_i})=\rho_g(v_{h,s_i})v_{g,s_i}\lambda_{s_i}=u_{g,h}v_{gh,s_i}\sigma_{s_i}(u_{g,h}^{\ast})\lambda_{s_i}=\Ad(u_{g,h})\circ\alpha_{gh}(\lambda_{s_i})$.
    This implies $\alpha_g\circ\alpha_h=\Ad(u_{g,h})\circ\alpha_{gh}$ on $M$.
\end{proof}

\begin{rem}\label{rem:localsection}
    If a non-empty open subset $U$ of $G$ satisfies the condition that $U\ni k\mapsto c(\cdot,\cdot,k)\in L^{\infty}(G\times G)$ is continuous with respect to SOT, then $U\ni k\mapsto \alpha_k\in\Aut(M)$ is continuous.
    In this case, $\kappa\colon G\to\Out(M)$ has a continuous local section to $\Aut(M)$.
\end{rem}

In the case $n=1$, since the group $\mathbb{F}_1=\mathbb{Z}$ is amenable, the crossed product $L^{\infty}(G)^{\mathbb{Z}}\rtimes\mathbb{Z}$ is an AFD factor of type $\II_1$.
This recovers Wassermann's original theorem.

\begin{thm}[Wassermann]
    Every 3-cocycle of a compact group $G$ can be realized as the lifting obstruction of a $G$-kernel on the AFD $\II_1$ factor.
\end{thm}

On the other hand, $\mathbb{F}_n$ is not inner-amenable if $n\ge 2$.
Thanks to this, we obtain the desired $G$-kernels on full factors.
We need a fact about Bernoulli crossed products.

\begin{prop}[cf.\ {\cite[Lemma 2.7]{VV15}}]\label{prop:fullness}
    Let $P$ be a non-trivial von Neumann algebra equipped with a faithful normal state.
    Then $P^{\otimes\mathbb{F}_n}\rtimes\mathbb{F}_n$ is a full factor for any $n\in\{2,3,\ldots,\infty\}$.
\end{prop}

\begin{thm}\label{thm:realization}
    There exists a full $\II_1$ factor $M$ such that every 3-cohomology class of a compact group $G$ can be realized as the obstruction of a $G$-kernel on $M$.
\end{thm}
\begin{proof}
    Let $P=L^{\infty}[0,1]$ endowed with the Lebesgue measure.
    Recall that every separable diffuse abelian von Neumann algebra is isomorphic to $P$.
    Define $M=P^{\otimes\mathbb{F}_2}\rtimes\mathbb{F}_2$, and then $M$ is a full factor of type $\II_1$ by Proposition \ref{prop:fullness} since $\mathbb{F}_2$ is not inner-amenable.

    Let $c\colon G^3\to\mathbb{T}$ be a measurable 3-cocycle of a compact group $G$.
    Theorem \ref{thm:main1} implies that there exists a $G$-kernel $\kappa$ on $M\cong L^{\infty}(G)^{\otimes\mathbb{F}_2}\rtimes\mathbb{F}_2$ such that $\Ob(\kappa)=[c]$.
    (When $G$ is finite, we replace $L^{\infty}(G)$ with $L^{\infty}(G)\otimes L^{\infty}[0,1]$ to ensure the isomorphism with $P$, and the same argument holds.)
\end{proof}

\subsection{String groups}\label{sec:String}

Throughout this subsection, we use the notation $\mathrm{H}_{\mathrm{grp}}^n(G;A)$, or more precisely, $\mathrm{H}_{\mathrm{grp,m}}^n(G;A)$ to denote the measurable group cohomology, distinguishing it from other cohomology theories.

Let $G$ be a compact, simple, and simply-connected Lie group.
It is known that its homotopy groups are
\[\pi_n(G)\cong\begin{cases}
    0&(n=0,1,2),\\
    \mathbb{Z}&(n=3).
\end{cases}\]
The 3-connected cover of $G$ is called the \emph{string group} of $G$ and denoted by $\String(G)$.
In particular, the 3-connected cover of $\Spin(n)$ for a large $n$ is denoted by $\String(n)$.
Several concrete constructions of $\String(G)$ are known (see \cite{St96, NSW13, KLW23}).
We will provide a new model for $\String(G)$ admitting a tractable topological group structure.

For any $\II_1$ factor $M$, the unitary group $\U(M)$ is contractible (\cite{Je25}) and hence $P\U(M)=\U(M)/\mathbb{T}$, equipped with SOT, models the Eilenberg--MacLane space $K(\mathbb{Z},2)$.
Our goal is to construct a fibration $P\U(M)\longinto\tilde{G}\longonto G$ and also to give an answer to \cite[Problem 3.10]{NSW13}, where $M$ is the full $\II_1$ factor from Theorem \ref{thm:realization}.

We know that $\mathrm{H^3_{grp}}(G;\mathbb{T})\cong\mathrm{H}^4(BG;\mathbb{Z})=[BG,K(\mathbb{Z},4)]\cong\Hom(\pi_3(G),\mathbb{Z})\cong\mathbb{Z}$.
Take a generator $[c]\in\mathrm{H_{grp}^3}(G;\mathbb{T})$ and let $\kappa\colon G\hookrightarrow\Out(M)$ be a $G$-kernel such that $\Ob(\kappa)=[c]$.
Then $\tilde{G}\coloneq\varepsilon_M^{-1}(\kappa(G))\subset\Aut(M)$ is a closed subgroup, and a projection $\varepsilon_M|_{\tilde{G}}\colon\tilde{G}\to G$ gives a principal $P\U(M)$-bundle over $G$ if the 3-cocycle $c$ satisfies the condition in Remark \ref{rem:localsection}.
We will show that this principal bundle corresponds to the generator of $\mathrm{H}^3(G;\mathbb{Z})\cong\mathbb{Z}$.
To this end, we introduce (semi-)simplicial spaces.
Our standard reference is \cite{Tu06}.

\begin{defn}
    A \emph{simplicial space} (resp.\ \emph{semi-simplicial space}) is a sequence $X_{\bullet}=(X_n)_{n\in\mathbb{N}}$ of topological spaces with a family of continuous maps $\tilde{f}\colon X_n\to X_m$ for every non-decreasing (resp.\ increasing) maps $f\colon[m]\to[n]$ satisfying the relation $\tilde{g}\circ\tilde{f}=\tilde{f\circ g}$, where $[n]$ denotes $\{0,\ldots,n\}$.
    When each $X_n$ is discrete, it is called a \emph{simplicial set}.
\end{defn}

$\Hom_{\Delta}(m,n)$ (resp.\ $\Hom_{\Delta'}(m,n)$) denotes the set of all non-decreasing (resp.\ increasing) maps from $[m]$ to $[n]$.
Let $\varepsilon_i\in\Hom_{\Delta'}(n-1,n)$ be the unique increasing map that avoids $i$, and $\eta_i^n\in\Hom_{\Delta}(n+1,n)$ be the unique non-decreasing surjective map that reaches $i$ twice.
In other words, a simplicial space is a family of continuous maps
\[\mathrm{d}_i^n\coloneqq\tilde{\varepsilon}_i^n\colon X_n\to X_{n-1},\quad\mathrm{s}_i^n\coloneqq\tilde{\eta}_i^n\colon X_n\to X_{n+1}\]
satisfying the following simplicial identities:
\[\begin{cases}
    \mathrm{d}_i^{n-1} \mathrm{d}_j^n = \mathrm{d}_{j-1}^{n-1} \mathrm{d}_i^n & \text{if } i<j,\\
    \mathrm{d}_i^{n+1} \mathrm{s}_j^n = \mathrm{s}_{j-1}^{n-1} \mathrm{d}_i^n & \text{if } i<j,\\
    \mathrm{d}_i^{n+1} \mathrm{s}_j^n = \id_{X_n} & \text{if } i=j,\, j+1,\\
    \mathrm{d}_i^{n+1} \mathrm{s}_j^n = \mathrm{s}_j^{n-1} \mathrm{d}_{i-1}^n & \text{if } i>j+1,\\
    \mathrm{s}_i^{n+1} \mathrm{s}_j^n = \mathrm{s}_{j+1}^{n+1} \mathrm{s}_i^n & \text{if } i\le j.
\end{cases}\]

\begin{defn}[cf.\ {\cite[Section 4]{Tu06}}]
    Let $X_{\bullet}$ be a simplicial space.
    A sequence of open cover $\mathfrak{U}_n=\{U_i^{(n)}\}_{\alpha \in I_n}$ of $X_n$ is said to be a \emph{semi-simplicial open cover} if $I_{\bullet}=(I_n)_n$ is a semi-simplicial set (namely, $\mathrm{d}_i^n\colon I_n\to I_{n-1}$ satisfies $\mathrm{d}_i^n\mathrm{d}_j^n=\mathrm{d}_{j-1}^n\mathrm{d}_i^n$ if $i<j$) and $\mathrm{d}_i^n(U_j^{(n)}) \subset U_{\mathrm{d}_i^n(j)}$.
\end{defn}

For a given cover $\mathfrak{U}_n$, we can form the semi-simplicial open cover $\sigma\mathfrak{U}_n=(\tilde{U}_{\lambda}^{(n)})_{\lambda\in\Lambda_n(I)}$.
First, we define the semi-simplicial set $\Lambda_n(I)$ by
\[\Lambda_n(I)=\prod_{m=0}^n\Map(\Hom_{\Delta'}(m,n),I_m),\quad[\tilde{g}(\lambda)](f)=\lambda(g\circ f)\]
for $g\in\Hom_{\Delta'}(n,n')$, $\lambda\in\Lambda_{n'}(I)$ and $f\in\Hom_{\Delta'}(m,n)$.
Next, for all $\lambda\in\Lambda_n(I)$, we let
\[\tilde{U}_{\lambda}^{(n)}=\bigcap_{m\le n}\bigcap_{f\in\Hom_{\Delta'}(m,n)}\tilde{f}^{-1}(U_{\lambda(f)}^{(m)}).\]
Then the open cover $(\tilde{U}_{\lambda}^{(n)})_{\lambda\in\Lambda_n(I)}$ is in fact a semi-simplicial open cover (cf.\ \cite[Section 4.1]{Tu06}).

\begin{defn}
    Let $X_{\bullet}$ be a semi-simplicial space, $\mathfrak{U}_{\bullet}$ a semi-simplicial open cover of $X_{\bullet}$, and $A^{\bullet}$ an abelian semi-simplicial sheaf.
    \begin{enumerate}
        \item Take a resolution
        \[\begin{tikzcd}
            0\arrow[r] &A^{\bullet}\arrow[r] &\mathcal{F}^{\bullet,0}\arrow[r,"\delta^{\bullet,0}"] &\mathcal{F}^{\bullet,1}\arrow[r,"\delta^{\bullet,1}"] &\cdots.
        \end{tikzcd}\]
        such that $\mathrm{H}^n(X_p;\mathcal{F}^{p,q})=0$ for all $n\ge 1$.
        Define the double complex $(\mathcal{F}^{p,q}(X_p),d^{p,q},\delta^{p,q})$ by $d^{p,q}=\sum_{i=0}^{p+1}(-1)^i(\mathrm{d}_i^{p+1})^{\ast}$ and let $\mathrm{H}^{\ast}(X_{\bullet};A^{\bullet})$ denote the cohomology group of its total complex.
        (cf.\ \cite[Section 4]{De74} and \cite[Section 7]{Tu06})
        \item Define a cochain complex $(\check{\mathrm{C}}_{\mathrm{ss}}^{\ast}(\mathfrak{U}_{\bullet};A^{\bullet}),d^{\ast})$ by
        \[\check{\mathrm{C}}_{\mathrm{ss}}^n(\mathfrak{U}_{\bullet};A^{\bullet})=\prod_{i\in I_n}A^n(U_i^{(n)}),\quad (d^nc)_j=\sum_{i=0}^{n+1}(-1)^i(\mathrm{d}_i^{n+1})^{\ast}c_{\mathrm{d}_i^{n+1}(j)}.\]
        One can check that $d^n\circ d^{n-1}=0$.
        Let $\check{\mathrm{H}}_{\mathrm{ss}}^{\ast}(\mathfrak{U}_{\bullet};A^{\bullet})$ be the cohomology groups of this complex.
        Define the \emph{\v{C}ech cohomology} of $(X_{\bullet},A^{\bullet})$ by
        \[\check{\mathrm{H}}^{\ast}(X_{\bullet},A^{\bullet})=\varinjlim_{\mathfrak{U}}\check{\mathrm{H}}_{\mathrm{ss}}^{\ast}(\sigma\mathfrak{U}_{\bullet};A^{\bullet})\]
        (cf. \cite[p.4733]{Tu06}).
    \end{enumerate}
\end{defn}

These two cohomology groups are isomorphic $\mathrm{H}^{\ast}(X_{\bullet};A^{\bullet})\cong\check{\mathrm{H}}^{\ast}(X_{\bullet};A^{\bullet})$ by \cite[Proposition 7.1]{Tu06}.
Note that the constant sheaf $\underline{A}$ on $X_{\bullet}$ has simplicial structure.

\begin{defn}
    The semi-simplicial space $B_{\bullet}G$ is defined by $B_nG\coloneq G^n$ with structure maps
    \begin{align*}
        \mathrm{d}_i^n\colon B_nG\to B_{n-1}G;&& \mathrm{d}_i^n(g_1,\ldots,g_n) &= \begin{cases}
            (g_2,\ldots,g_n) & \text{if } i=0,\\
            (g_1,\ldots,g_ig_{i+1},\ldots,g_n) & \text{if } 0<i<n,\\
            (g_1,\ldots,g_{n-1}) & \text{if } i=n,
        \end{cases}\\
        \mathrm{s}_i^n\colon B_nG\to B_{n+1}G;&& 
        \mathrm{s}_i^n(g_1,\ldots,g_n) &= (g_1,\ldots,g_i,e,g_{i+1},\ldots, g_n).
    \end{align*}
    The semi-simplicial space $E_{\bullet}G$ is defined by $E_nG\coloneq G^{n+1}$ with structure maps
    \begin{align*}
        \mathrm{d}_i^n\colon E_nG\to E_{n-1}G;&& \mathrm{d}_i^n(g_0,\ldots,g_n) &= \begin{cases}
            (g_0,\ldots,g_ig_{i+1},\ldots,g_n) & \text{if } 0\le i<n,\\
            (g_0,\ldots,g_{n-1}) & \text{if } i=n,
        \end{cases}\\
        \mathrm{s}_i^n\colon E_nG\to E_{n+1}G;&& 
        \mathrm{s}_i^n(g_0,\ldots,g_n) &= (g_0,\ldots,g_i,e,g_{i+1},\ldots, g_n).
    \end{align*}
    We may consider that $G$ is a trivial simplicial subspace of $E_{\bullet}G$ under $G\ni g\mapsto (g,e,\ldots,e)\in E_nG$.
    Forgetting $g_0$ gives the simplicial map $\pi\colon E_{\bullet}G\to B_{\bullet}G$.
\end{defn}

\begin{lem}
    $\mathrm{H}^{\ast}(E_{\bullet}G;\underline{\mathbb{T}})=0$.
\end{lem}
\begin{proof}
    Let $\mathcal{F}^{\bullet,q}$ be a resolution of the simplicial sheaf $\underline{\mathbb{T}}$ on $B_{\bullet}G$ such that each $\mathcal{F}^{p,q}$ is flabby.
    Regard $\mathcal{G}^{p,q}\coloneqq\mathcal{F}^{p+1,q}$ as a simplicial sheaf on $E_pG=G^{p+1}$.
    We claim that the map
    \[h^{p,q}\coloneqq(\mathrm{s}_0^{p-1})^{\ast}\colon\mathcal{G}^{p,q}(E_pG)\to\mathcal{G}^{p-1,q}(E_{p-1}G),\]
    gives the chain homotopy from $\id$ to $0$ of the total complex.
\end{proof}

Wigner's definition of semi-simplicial group cohomology is $\mathrm{H}_{\mathrm{grp,ss}}^{\ast}(G;A)=\check{\mathrm{H}}^{\ast}(B_{\bullet}G;\underline{A})$.
It is known that this cohomology is isomorphic to Moore's measurable cohomology (see \cite{AM13} and \cite[Proposition 6.2]{Tu06}).
The explicit construction of this isomorphism is given in \cite[Remark 6.3]{Tu06} as follows:
Fix an open cover $\mathfrak{U}_{\bullet}$.
Choose measurable maps $\theta_{\bullet}\colon B_{\bullet}G\to I_{\bullet}$ such that $(g_1,\ldots,g_n)\in U_{\theta_n(g_1,\ldots,g_n)}^{(n)}$, and set a semi-simplicial map $\lambda_n\colon B_nG\to\Lambda_n(I)$ by $[\lambda_n(g_1\ldots,g_n)](f)=\theta_m(\tilde{f}(g_1,\ldots,g_n))$ for every $f\in\Hom_{\Delta'}(m,n)$.
Now the correspondence $\check{\mathrm{C}}_{\mathrm{ss}}^n(\sigma\mathfrak{U}_{\bullet};\underline{A})\ni\varphi\mapsto c\in\mathrm{C}_{\mathrm{grp,m}}^n(G;A)$ defined by
\[c(g_1,\ldots,g_n)=\varphi_{\lambda_n(g_1,\ldots,g_n)}(g_1,\ldots,g_n)\]
induces the isomorphism of cohomology groups.

\begin{prop}
    Let $G$ be a 2-connected compact group, $\kappa\colon G\to\Out(M)$ a $G$-kernel on a full factor $M$.
    Suppose that $\kappa$ has the continuous local section to $\Aut(M)$.
    Then the obstruction $\Ob(\kappa)\in\mathrm{H_{grp}^3}(G;\mathbb{T})$ corresponds to the Dixmier--Douady invariant in $\mathrm{H}^3(G;\mathbb{Z})$ of $P\U(M)$-bundle $\tilde{G}\to G$ with respect to isomorphisms
    \[\mathrm{H^3_{grp}}(G;\mathbb{T})\cong\mathrm{H}^4(BG;\mathbb{Z})=[BG,K(\mathbb{Z},4)]\cong\Hom(\pi_3(G),\mathbb{Z})\cong\mathrm{H}^3(G;\mathbb{Z}).\]
\end{prop}
\begin{proof}
    Step1:
    We observe the isomorphisms $\mathrm{H_{grp,m}^3}(G;\mathbb{T})\cong\mathrm{H_{grp,ss}^3}(G;\mathbb{T})=\check{\mathrm{H}}^3(B_{\bullet}G;\underline{\mathbb{T}})$.
    There exists $\Obss(\kappa)\in\check{\mathrm{H}}^3(B_{\bullet}G;\underline{\mathbb{T}})$ corresponding to $\Ob(\kappa)$, defined by the following steps.
    Here $\mathfrak{U}_{\bullet}'$ is sufficiently fine open cover and put $\mathfrak{U}_{\bullet}=\sigma\mathfrak{U}_{\bullet}'$.
    \begin{enumerate}[label=(\roman*)]
        \item Let $\phi_i\colon U_i^{(1)}\to\Aut(M)$ be a continuous section with $\phi_i(e)=\id_M$ if $e\in U_i^{(1)}$;
        \item Let $\psi_j\colon U_j^{(2)}\to\U(M)$ be a continuous lift of
        \[\bar{\psi}_j\colon U_j^{(2)}\ni(g,h)\mapsto\phi_{\mathrm{d}_2j}(g)\circ\phi_{\mathrm{d}_0j}(h)\circ\phi_{\mathrm{d}_1j}(gh)^{-1}\in\Int(M)=P\U(M)\]
        such that $\psi_j(g,h)=1$ if either $g$ or $h$ is $e$;
        \item Define $\varphi_l\colon U_l^{(3)}\to\mathbb{T}$ by
        \[\varphi_l(g,h,k)=[\phi_{\mathrm{d}_2\mathrm{d}_3l}(g)](\psi_{\mathrm{d}_0l}(h,k))\cdot\psi_{\mathrm{d}_2l}(g,hk)\cdot\psi_{\mathrm{d}_1l}(gh,k)^{-1}\cdot\psi_{\mathrm{d}_3l}(g,h)^{-1}.\]
        Put $\Obss(\kappa)\coloneqq\varphi=(\varphi_l)_l$.
    \end{enumerate}
    Then $\Obss(\kappa)$ maps to $\Ob(\kappa)$.
    Indeed, two maps $\alpha\colon G\ni g\mapsto\phi_{\lambda_1(g)}(g)\in\Aut(M)$ and $u\colon G\times G\ni(g,h)\mapsto \psi_{\lambda_2(g,h)}(g,h)\in\U(M)$ satisfy $\varepsilon_M\circ\alpha=\kappa$ and $\Ad(u_{g,h})\circ\alpha_{gh}=\alpha_g\circ\alpha_h$.
    Thus, $[c]=\Ob(\kappa)$ is
    \begin{align*}
        c(g,h,k) &= \alpha_g(u_{h,k})u_{g,hk}u_{gh,k}^{\ast}u_{g,h}^{\ast}\\
        &= [\phi_{\lambda_1(g)}(g)](\psi_{\lambda_2(h,k)}(h,k))\cdot\psi_{\lambda_2(g,hk)}(g,hk)\cdot\psi_{\lambda_2(gh,k)}(gh,k)^{-1}\cdot\psi_{\lambda_2(g,h)}(g,h)^{-1}\\
        &= [\phi_{\mathrm{d}_2\mathrm{d}_3\lambda_3(g,h,k)}(g)](\psi_{\mathrm{d}_0\lambda_3(g,h,k)}(h,k))\cdot\psi_{\mathrm{d}_2\lambda_3(g,h,k)}(g,hk)\\
        &\mathrel{\phantom{=}} \cdot\psi_{\mathrm{d}_1\lambda_3(g,h,k)}(gh,k)^{-1}\cdot\psi_{\mathrm{d}_3\lambda_3(g,h,k)}(g,h)^{-1}\\
        &= \varphi_{\lambda_3(g,h,k)}(g,h,k).
    \end{align*}

    Step2:
    Next we consider $\check{\mathrm{H}}^3(B_{\bullet}G;\underline{\mathbb{T}})\cong\check{\mathrm{H}}^2(G;\underline{\mathbb{T}})\cong\mathrm{H}^3(G;\mathbb{Z})$.
    Recall $\pi\colon E_{\bullet}G\to B_{\bullet}G$ is defined by $(g_0,\ldots,g_n)\mapsto(g_1,\ldots,g_n)$, hence $G\subset E_{\bullet}G$ is the inverse image of $\pt\subset B_{\bullet}G$ under $\pi$.
    Thus, $\pi$ induces $\pi^{\ast}\colon\check{\mathrm{H}}^{\ast}(B_{\bullet}G,\pt;\underline{\mathbb{T}})\to\check{\mathrm{H}}^{\ast}(E_{\bullet}G,G;\underline{\mathbb{T}})$.
    We have the long exact sequence
    \[\begin{tikzcd}
        \cdots\to\check{\mathrm{H}}^{\ast}(E_{\bullet}G;\underline{\mathbb{T}})\arrow[r] &\check{\mathrm{H}}^{\ast}(G;\underline{\mathbb{T}})\arrow[r,"\partial"] &\check{\mathrm{H}}^{\ast+1}(E_{\bullet}G,G;\underline{\mathbb{T}})\arrow[r] &\check{\mathrm{H}}^{\ast+1}(E_{\bullet}G;\underline{\mathbb{T}})\to\cdots.
    \end{tikzcd}\]
    Since $\check{\mathrm{H}}^{\ast}(E_{\bullet}G;\underline{\mathbb{T}})=0$, $\partial$ is an isomorphism.
    There exists a commuting diagram:
    \[\begin{tikzcd}
        \mathrm{H}^3(B_{\bullet}G,\pt;\underline{\mathbb{T}})\arrow[r,"\pi^{\ast}"]\arrow[d,"\delta"] & \mathrm{H}^3(E_{\bullet}G,G;\underline{\mathbb{T}})\arrow[d,"\delta"] &\mathrm{H}^2(G;\mathbb{T})\arrow[l,"\partial"']\arrow[d,"\delta"]\\
        \mathrm{H}^4(B_{\bullet}G,\pt;\mathbb{Z})\arrow[r,"\pi^{\ast}"]\arrow[d,"\cong"] & \mathrm{H}^4(E_{\bullet}G,G;\mathbb{Z})\arrow[d,"\cong"] &\mathrm{H}^3(G;\mathbb{Z})\arrow[l,"\partial"']\arrow[d,"\cong"]\\
        \Hom(\pi_4(BG,\pt),\mathbb{Z})\arrow[r,"\pi_{\ast}\circ-"] &\Hom(\pi_4(EG,G),\mathbb{Z})\arrow[r] &\Hom(\pi_3(G),\mathbb{Z}).
    \end{tikzcd}\]
    The isomorphism $\partial^{-1}\circ\pi^{\ast}\colon\mathrm{H}^3(B_{\bullet}G;\underline{\mathbb{T}})\cong\mathrm{H}^2(G;\underline{\mathbb{T}})$ is explicitly constructed as follows.
    Let $\xi\in\mathrm{H}^3(B_{\bullet}G,\pt;\underline{\mathbb{T}})$.
    We find $\eta\in\mathrm{C}^2(E_{\bullet}G;\underline{\mathbb{T}})$ such that $\pi^{\ast}\xi=d\eta$.
    Restrict it on $G\subset E_{\bullet}G$, we have $d\eta|_G=\pi^{\ast}\xi|_{\pt}=0$ and hence $\eta|_G$ is a cocycle.
    $\xi\mapsto\eta|_G$ gives this isomorphism.

    Now, for $\xi=\Obss(\kappa)$, we can take $\eta$ as
    \[\eta_l(g,h,k)=\varphi_l(g,h,k).\]
    Restrict $\eta$ on $G\subset E_{\bullet}G$ and we obtain a \v{C}ech cocycle
    \[\eta_l(g,e,e)=\varphi_l(g,e,e)=\psi_{d_2l}(g,e)\cdot\psi_{d_1l}(g,e)^{-1}\cdot\psi_{d_3l}(g,e)^{-1}\]
    over a topological space $G$.
    Now, $\bar{\psi}_j(g,e)=\phi_{d_2j}(g)\circ\phi_{d_1j}(g)^{-1}\in P\U(M)$ is nothing but a transition function of the principal $P\U(M)$-bundle $\tilde{G}$ over $G$.
    Therefore, $\varphi_l(g,e,e)$ is the derivation of the lift $\psi_j(g,e)$ of $\bar{\psi}_j(g,e)$, which is precisely the element of $\check{\mathrm{H}}^2(G;\underline{\mathbb{T}})$ associated with the principal $P\U(M)$-bundle $\tilde{G}$.
\end{proof}

Recall that the \emph{gauge group} $\Gau(P)$ of a principal bundle $P\to X$ is the group of bundle automorphisms of $P$ that induces the identity on the base space $X$.

\begin{lem}
    In this setting, we can choose a representing 3-cocycle $c$ for the generator of $\mathrm{H_{grp,m}^3}(G;\mathbb{T})\cong\mathbb{Z}$ that satisfies the condition in Remark \ref{rem:localsection}.
\end{lem}
\begin{proof}
    Thanks to the NSW-model (\cite{NSW13}) for $\String(G)$, we find that the following data.
    \begin{itemize}
        \item A principal $P\U(\mathcal{H})$-bundle $P\to G$ corresponding to the generator of $\mathrm{H}^3(G;\mathbb{Z})\cong\mathbb{Z}$, where $\mathcal{H}$ denotes the separable infinite-dimensional Hilbert space;
        \item An extension of topological groups $\Gau(P)\longinto\tilde{G}\longonto G$;
        \item A central extension $\mathbb{T}\longinto\tilde{\Gau}(P)\longonto\Gau(P)$.
    \end{itemize}
    Since the extension $\tilde{G}\to G$ is locally trivial, there exist a measurable section $\phi\colon G\to\tilde{G}$ and a open dense subset $U\subset G$ such that $\phi$ is continuous on $U$.
    Then the map $\bar{\psi}\colon G\times G\ni(g,h)\mapsto \phi(g)\phi(h)\phi(gh)^{-1}\in\tilde{G}$ takes values in $\Gau(P)$.
    We can choose a measurable section $s\colon\Gau(P)\to\tilde{\Gau}(P)$ such that the lift $\psi=s\circ\bar{\psi}$ is continuous on the open dense subset $V=\{(g,h)\in G^2\mid g,h,gh\in U\}$ (by replacing $U$ with a smaller one if necessary).
    Note that for $\alpha\in\tilde{G}$, $\Ad{\alpha}\in\Aut(\Gau(P))$ extends uniquely to an automorphism of $\tilde{\Gau}(P)$.
    We now define $c$ by $c(g,h,k)=[\Ad{\phi(g)}](\psi(h,k))\psi(g,hk)\psi(gh,k)^{-1}\psi(g,h)^{-1}$, which satisfies the 3-cocycle identity and is representing the generator since the principal bundle $P$ corresponds to the generator of $\mathrm{H}^3(G;\mathbb{Z})\cong\mathrm{H_{grp,m}^3}(G;\mathbb{T})$.
    By construction, the function $c$ is continuous on the set
    \[W=\{(g,h,k)\in G^3\mid g,h,k,gh,hk,ghk\in U\},\]
    which is open and dense in $G^3$.
    
    Fix $k\in U$ and a net $k_i\in U$ converging to $k$.
    Let $W_k$ denote the slice $\{(g,h)\in G^2\mid (g,h,k)\in W\}$.
    Since $W_k=\{(g,h)\mid g,h,gh,hk,ghk\in U\}$ is an intersection of open dense subsets, it is also open and dense in $G^2$, and thus it has full measure.
    For any $\varepsilon>0$, there exist a compact subset $A\subset W_k$ and a neighborhood $U'$ of $k$ such that $|G^2\setminus A|<\varepsilon$ and $A\times U'\subset W$.
    Then the convergence $c(g,h,k_i)\to c(g,h,k)$ as $i\to\infty$ is uniform on $(g,h)\in A$.
    For sufficiently large $i$, we have $|c(g,h,k_i)-c(g,h,k)|^2<\varepsilon$ for every $(g,h)\in A$ and hence
    \begin{align*}
        \|c(\cdot,\cdot,k_i)-c(\cdot,\cdot,k)\|_2^2 &= 4|G^2\setminus A|+\int_A|c(g,h,k_i)-c(g,h,k)|^2\,d(g,h)\\
        &\le 4\varepsilon+\int_A\varepsilon\,d(g,h)\\
        &< 5\varepsilon.
    \end{align*}
    Thus, $c$ satisfies the requirement in Remark \ref{rem:localsection}.
\end{proof}

As a summary, we have the following corollary.

\begin{cor}
    Let $G$ be a compact, simple, and simply-connected Lie group.
    Then there exist a full $\II_1$ factor $M$ and a principal $P\U(M)$-bundle $\tilde{G}\to G$ corresponding to the generator of $\mathrm{H}^3(G;\mathbb{Z})\cong\mathbb{Z}$.
    In particular, this provides the short exact sequence of Polish groups $P\U(M)\longinto\tilde{G}\longonto G$, which is a Hurewicz fibration.
    The group $\tilde{G}$ is a Polish group model for the string group $\String(G)$.
\end{cor}

\section{The Rohlin property}

In what follows, let $G$ denote a compact group.

\begin{defn}
    An action or cocycle action $(\alpha,u)\colon G\curvearrowright M$ is said to have the \emph{Rohlin property} if there exists a $G$-equivariant embedding $L^{\infty}(G)\hookrightarrow M_{\omega,\alpha}$.
    Such an action is simply called a \emph{Rohlin action}.
\end{defn}

It follows immediately from the definition that $\alpha\otimes\beta$ has the Rohlin property if either $\alpha$ or $\beta$ does.
Note that the Rohlin property is preserved under conjugation.
Indeed, $M_{\omega,\alpha}=M_{\omega,\alpha'}$ if $(\alpha',u')$ is a unitary perturbation of $(\alpha,u)$.

\begin{rem}
    In \cite{Sh14}, the Rohlin property of a cocycle action $(\alpha,u)$ of a locally compact abelian group $G$ on $M$ is defined by the existence of a unitary $v_p\in\U(M_{\omega,\alpha})$ satisfying $\alpha_g(v_p)=\inn{g,p}v_p$ for all $g\in G$ and $p\in\hat{G}$.
    For compact abelian groups, these two definitions are equivalent.
    % In fact, an action of locally compact abelian group $G$ has the Rohlin property in the sense of \cite{Sh14} if and only if every compact subquotient has the Rohlin property.
\end{rem}

Let $(\alpha,u)\colon G\curvearrowright M$ be a cocycle action on a factor $M$ with the Rohlin property.
We consider $M\otimes L^{\infty}(G)\cong M\vee L^{\infty}(G)\subset M_{\alpha}^{\omega}$ such that $[\alpha_g(x)](h)=\alpha_g(x(g^{-1}h))$ for $x\in M\otimes L^{\infty}(G)$.
This observation plays a central role in the present and the next section.

\subsection{Classification of Rohlin actions}

We will classify Rohlin actions following \cite{MT16}.
Since we deal only with compact groups, approximately 1-cohomology vanishing and 2-cohomol\-ogy vanishing theorems are not needed here.
These results are shown in the next section.

\begin{lem}\label{lem:step1}
    Let $\alpha,\beta\colon G\curvearrowright M$ be actions of a compact group $G$ on a factor $M$.
    Suppose that $\alpha$ has the Rohlin property and $\beta_g\circ\alpha_g^{-1}\in\clInt(M)$ for all $g\in G$.
    Then for any $\varepsilon>0$ and finite subset $\Phi\Subset(M_{\ast})_1$, there exists $\theta\in\Int(M)$ such that
    \[\int_G\|\theta\circ\alpha_g\circ\theta^{-1}\circ\beta_g^{-1}(\varphi)-\varphi\|\,dg<\varepsilon,\quad
    \|\theta(\varphi)-\varphi\|<\int_G\|\beta_g\circ\alpha_g^{-1}(\varphi)-\varphi\|\,dg+\varepsilon\]
    for all $\varphi\in\Phi$.
\end{lem}
\begin{proof}
    By assumption, there exists a Borel map $v\colon G\to\U(M)$ such that
    \[\sup_{g,h,k\in G}\max_{\varphi\in\Phi}\|\Ad{v_g}\circ\alpha_h\circ\beta_k(\varphi)-\beta_g\circ\alpha_{g^{-1}h}\circ\beta_k(\varphi)\|<\frac{\varepsilon}{2}.\]
    Consider a unitary representing sequence $(w^{\nu})_{\nu}$ for the unitary $W\in M\otimes L^{\infty}(G)\subset M^{\omega}_{\alpha}$ given by $W(h)=v_h$.
    We shall show that $\theta=\Ad{w^{\nu}}$ satisfies the condition for a sufficiently large $\nu$.

    From Lemma \ref{lem:AdW}, it follows that
    \begin{align*}
        \|\Ad{W}\circ\alpha_g\circ\Ad{W^{\ast}}\circ\beta_g^{-1}(\varphi^{\omega})-\varphi^{\omega}\| &= \|\Ad(W\alpha_g(W^{\ast}))\circ\alpha_g\circ\beta_g^{-1}(\varphi^{\omega})-\psi^{\omega}\|\\
        &\le \int_G\|\Ad([W\alpha_g(W^{\ast})](h))\circ\alpha_g\circ\beta_g^{-1}(\varphi)-\varphi\|\,dh\\
        &= \int_G\|\Ad(v_h\alpha_g(v_{g^{-1}h}^{\ast}))\circ\alpha_g\circ\beta_g^{-1}(\varphi)-\varphi\|\,dh,\\
        \|\Ad(v_h\alpha_g(v_{g^{-1}h}^{\ast}))\circ\alpha_g\circ\beta_g^{-1}(\varphi)-\varphi\| &= \|\Ad(v_h)\circ\alpha_g\circ\Ad(v_{g^{-1}h}^{\ast})\circ\beta_g^{-1}(\varphi)-\varphi\|\\
        &\le \|\Ad(v_{g^{-1}h}^{\ast})\circ\beta_g^{-1}(\varphi)-\alpha_{g^{-1}h}\circ\beta_h^{-1}(\varphi)\|\\
        &\mathrel{\phantom{\le}} +\|\Ad(v_h)\circ\alpha_h\circ\beta_h^{-1}(\varphi)-\varphi\|\\
        &< \varepsilon.
    \end{align*}
    Thus we have 
    \begin{align*}
        &\mathrel{\phantom{=}}\lim_{\nu\to\omega}\int_G\|\Ad(w^{\nu})\circ\alpha_g\circ\Ad(w^{\nu})^{-1}\circ\beta_g^{-1}(\varphi)-\varphi\|\,dg\\
        &= \int_G\lim_{\nu\in\omega}\|\Ad(w^{\nu})\circ\alpha_g\circ\Ad(w^{\nu})^{-1}\circ\beta_g^{-1}(\varphi)-\varphi\|\,dg\\
        &= \int_G\|\Ad{W}\circ\alpha_g\circ\Ad{W^{\ast}}\circ\beta_g^{-1}(\varphi)-\varphi\|\,dg<\varepsilon.
    \end{align*}

    In addition, similar calculation shows that
    \begin{align*}
        \lim_{\nu\to\omega}\|\Ad{w^{\nu}}(\varphi)-\varphi\| &= \|\Ad{W}(\varphi^{\omega})-\varphi^{\omega}\|\\
        &\le \int_G\|\Ad{v_h}(\varphi)-\varphi\|\,dh\\
        &\le \sup_{h\in G}\|\Ad{v_h}(\varphi)-\beta_h\circ\alpha_h^{-1}(\varphi)\|+\int_G\|\beta_h\circ\alpha_h^{-1}(\varphi)-\varphi\|\,dh\\
        &= \int_G\|\beta_g\circ\alpha_g^{-1}(\varphi)-\varphi\|\,dg+\varepsilon.
    \end{align*}
\end{proof}

Here is the complete classification result for Rohlin actions.

\begin{thm}\label{thm:classification}
    Let $\alpha,\beta\colon G\curvearrowright M$ be Rohlin actions of a compact group $G$ on a factor $M$.
    If $\beta_g\circ\alpha_g^{-1}\in\clInt(M)$ for all $g\in G$, then they are approximately inner conjugate, i.e., there exists $\sigma\in\clInt(M)$ such that $\sigma\circ\alpha\circ\sigma^{-1}=\beta$.
\end{thm}
\begin{proof}
    The proof is by using Bratteli--Elliott--Evans--Kishimoto intertwining argument in the same way as in \cite{MT16}.
    Let $(\varphi_i)_{i=0}^{\infty}\subset (M_{\ast})_1$ be a dense sequence. 
    Set $\Phi_n=(\varphi_i)_{i=0}^n$, $\gamma^{(-1)}=\alpha$, $\gamma^{(0)}=\beta$, $\tilde{\Phi}_{-1}=\Phi_0$ and $\tilde{\theta}_{-1}=\tilde{\theta}_{-2}=\id$.
    Inductively, for any $n\ge 0$, Lemma \ref{lem:step1} provides $\theta_n\in\Int(M)$ so that
    \[\int_G\|\theta_n\circ\gamma_g^{(n-1)}\circ\theta_n^{-1}\circ(\gamma_g^{(n)})^{-1}(\varphi)-\varphi\|\,dg<2^{-n-1},\]
    \[\|\theta_n(\varphi)-\varphi\|<\int_G\|\gamma_g^{(n)}\circ(\gamma_g^{(n-1)})^{-1}(\varphi)-\varphi\|\,dg+2^{-n}\]
    for all $\varphi\in\tilde{\Phi}_n\coloneqq\Phi_n\cup\tilde{\Phi}_{n-1}\cup\tilde{\theta}_{n-1}(\tilde{\Phi}_{n-1})$.
    Put $\gamma^{(n+1)}=\theta_n\circ\gamma^{(n-1)}\circ\theta_n^{-1}$ and $\tilde{\theta}_n=\theta_n\circ\tilde{\theta}_{n-2}$.
    Note that for any $\varphi\in\tilde{\Phi}_{n-2}$, it follows that
    \[\|\tilde{\theta}_n(\varphi)-\tilde{\theta}_{n-2}(\varphi)\|<\int_G\|\gamma_g^{(n)}\circ(\gamma_g^{(n-1)})^{-1}(\tilde{\theta}_{n-2}(\varphi))-\tilde{\theta}_{n-2}(\varphi)\|\,dg+2^{-n}<2^{-n+1}\]
    and
    \[\|\tilde{\theta}_n^{-1}(\varphi)-\tilde{\theta}_{n-2}^{-1}(\varphi)\|=\|\theta_n(\varphi)-\varphi\|<2^{-n+1}\]
    since both $\varphi$ and $\tilde{\theta}_{n-2}(\varphi)$ belong to $\tilde{\Phi}_{n-1}$.
    Thus there exist $\lim_{n\to\infty}\tilde{\theta}_{2n}=\sigma_0\in\clInt(M)$ and $\lim_{n\to\infty}\tilde{\theta}_{2n+1}=\sigma_1\in\clInt(M)$.
    Now one can check that $\sigma_0\circ\alpha\circ\sigma_0^{-1}=\sigma_1\circ\beta\circ\sigma_1^{-1}$.
    Letting $\sigma=\sigma_1^{-1}\circ\sigma_0$ completes the proof.
\end{proof}

\subsection{Characterization of Rohlin actions}

\begin{prop}\label{prop:minRohlin}
    Let $\alpha\colon G\curvearrowright R$ be an action of a compact group $G$ on the AFD $\II_1$ factor.
    If $\alpha$ is minimal, then $\alpha$ has the Rohlin property.
\end{prop}
\begin{proof}
    By Theorem \ref{thm:uniqmin}, it suffices to find a minimal action with the Rohlin property.
    Consider $\alpha\coloneqq\mathcal{L}^{\otimes\infty}\colon G\curvearrowright L^{\infty}(G)^{\otimes\infty}$ and $\sigma\colon \mathfrak{S}_{\infty}\curvearrowright L^{\infty}(G)^{\otimes\infty}$, where $\mathfrak{S}_{\infty}$ is the symmetric group of finite permutations and acts on indices.
    These two actions are commuting and hence $\alpha$ extends to $G\curvearrowright (L^{\infty}(G)^{\otimes\infty}\rtimes_{\sigma}\mathfrak{S}_{\infty})$.
    Since $L^{\infty}(G)^{\otimes\infty}$ and $\mathfrak{S}_{\infty}$ are both amenable, so is the crossed product $L^{\infty}(G)^{\otimes\infty}\rtimes \mathfrak{S}_{\infty}$.
    In addition, the action $\sigma$ is free and ergodic, thus we have $L^{\infty}(G)^{\otimes\infty}\rtimes \mathfrak{S}_{\infty}\cong R$.

    The action $\alpha$ has the Rohlin property:
    Let $\iota^n\colon L^{\infty}(G)\hookrightarrow (L^{\infty}(G)^{\otimes\infty}\rtimes \mathfrak{S}_{\infty})$ be an embedding into the $n$-th tensor component of $L^{\infty}(G)$.
    Then the map $L^{\infty}(G)\ni f\mapsto (\iota^{\nu}(f))_{\nu}\in R_{\omega}$ defines an inclusion into the central sequence algebra.
    Since $\alpha_g\circ\iota^{\nu}=\iota^{\nu}\circ\mathcal{L}_g$, the sequence of maps $g\mapsto \alpha_g(\iota^{\nu}(f))$ is equicontinuous.
    Hence, $L^{\infty}(G)$ is equivariantly embedded into $R_{\omega,\alpha}$.

    Minimality of the action $\alpha$:
    Since $R^{\alpha}$ contains the subalgebra $L(\mathfrak{S}_{\infty})$, we have $(R^{\alpha})'\cap R\subset L(\mathfrak{S}_{\infty})'\cap (L^{\infty}(G)^{\otimes\infty}\rtimes \mathfrak{S}_{\infty})=\mathbb{C}$.
    This completes the proof.
\end{proof}

\begin{rem}
    It is possible to prove this proposition without using the uniqueness of minimal actions.
    Here is its sketch:
    the minimal action $\alpha$ is dual, namely, $R=R^G\rtimes_{\beta}{\hat{G}}$ for some coaction $\beta\colon\hat{G}\curvearrowright R^G$ and $\alpha=\hat{\beta}$ (\cite[Corollary 7.4]{MT07}).
    Since $R^G$ is also an AFD $\II_1$ factor, we have that $\beta$ is approximately representable, which means that $\alpha$ has the Rohlin property.
    See also \cite[Section 3.4]{MT10}.
\end{rem}

\begin{prop}\label{prop:str.outer}
    Let $\alpha\colon G\curvearrowright M$ be an action of a compact group $G$ on a von Neumann algebra $M$ (not necessarily a factor).
    Suppose that $\alpha$ has the Rohlin property.
    Then the following statements hold.
    \begin{enumerate}
        \item The action $\alpha$ is strictly outer.
        \item Its canonical extension $\tilde{\alpha}$ has the Rohlin property.
    \end{enumerate}
\end{prop}
\begin{proof}
    The proof is similar to \cite[Remark 4.12 and Corollary 4.13]{MT16}.
    
    (1):
    Let $x\in M'\cap (M\rtimes_{\alpha}G)$.
    For any $\pi\in\hat{G}$, $\alpha_{\pi}$ is approximated by $w(\cdot\otimes 1)w^{\ast}$ with some unitary $w\in M\otimes\mathbb{B}(\mathcal{H}_{\pi})$.
    Hence $\alpha_{\pi}(x)=x\otimes 1$, which means $x\in (M\rtimes_{\alpha}G)^{\hat{\alpha}}=M$.

    (2):
    The inclusions $L^{\infty}(G)\subset M_{\omega}\subset\tilde{M}_{\omega}$ are both equivariant.
    For the second inclusion, see \cite[Corollary 1.9]{MT17}.
\end{proof}

\begin{cor}
    Let $\alpha\colon G\curvearrowright R$ be an action of a compact group on the AFD $\II_1$ factor.
    The following conditions are equivalent.
    \begin{enumerate}
        \item $\alpha$ is minimal;
        \item $\alpha$ is strictly outer;
        \item $\alpha$ has the Rohlin property.
    \end{enumerate}
\end{cor}
\begin{proof}
    (1)$\Longleftrightarrow$(2) is shown in \cite[Proposition 6.2]{Va01}.
    (1)$\Longrightarrow$(3) follows from Proposition \ref{prop:minRohlin}, and (3)$\Longrightarrow$(1) follows from Proposition \ref{prop:str.outer}.
\end{proof}

Combining this and the classification results, we obtain the characterization on McDuff factors.

\begin{cor}
    Let $\beta\colon G\curvearrowright M$ be an action of a compact group $G$ on a McDuff factor $M$.
    Then the following conditions are equivalent.
    \begin{enumerate}
        \item $\beta$ is conjugate to $\beta\otimes\alpha$ for any action $\alpha\colon G\curvearrowright R$;
        \item $\beta$ is conjugate to $\beta\otimes\gamma$ for \emph{the} minimal action $\gamma\colon G\curvearrowright R$;
        \item $\beta$ has the Rohlin property.
    \end{enumerate}
\end{cor}
\begin{proof}
    (1)$\Longrightarrow$(2)$\Longrightarrow$(3) is easy.

    (3)$\Longrightarrow$(1):
    Since $M$ is McDuff, there exist an isomorphism $\Theta\colon M\to M\otimes R$ and a sequence of unitaries $w_n\in\U(M\otimes R)$ such that $\|\Ad{w_n}\circ\Theta(\varphi)-\varphi\otimes\tau\|$ tends to zero for all $\varphi\in M_{\ast}$.
    Fix $g\in G$ and we simply write $\alpha=\alpha_g$ and $\beta=\beta_g$.
    Then it suffices to show that $\beta^{-1}\circ\Theta^{-1}\circ(\beta\otimes\alpha)\circ\Theta\in\clInt(M)$.

    Set $\tilde{w}_n=(\beta\otimes\alpha)(w_n^{\ast})w_n$.
    Then we have
    \begin{align*}
        \|\Ad{\tilde{w}_n}\circ\Theta\circ\beta(\varphi)-(\beta\otimes\alpha)\circ\Theta(\varphi)\| &= \|\Ad{w_n}\circ\Theta\circ\beta(\varphi)-(\beta\otimes\alpha)\circ\Ad{w_n}\circ\Theta(\varphi)\|\\
        &\le \|\Ad{w_n}\circ\Theta(\beta(\varphi))-\beta(\varphi)\otimes\tau\|\\
        &\mathrel{\phantom{\le}} +\|(\beta\otimes\alpha)(\varphi\otimes\tau-\Ad{w_n}\circ\Theta(\varphi))\|\\
        &\to 0
    \end{align*}
    for every $\varphi\in M_{\ast}$.
    Therefore $\beta^{-1}\circ\Theta^{-1}\circ(\beta\otimes\alpha)\circ\Theta=\lim_{n\to\infty}\Ad{\tilde{w}_n}$ is approximately inner.
    Theorem \ref{thm:classification} completes the proof.
\end{proof}

Next we treat actions on AFD factors of type $\III$.
Proposition \ref{prop:str.outer} implies that $\tilde{\alpha}$ is strictly outer if $\alpha$ has the Rohlin property.
We want to consider the converse of this statement (see \cite[Conjecture 8.3]{MT16}).

\begin{prop}
    Let $\alpha\colon G\curvearrowright M$ be an action of a compact group $G$ on an AFD factor $M$ of type $\III$.
    \begin{enumerate}
        \item If the Connes--Takesaki module $\mod(\alpha)$ is trivial and its canonical extension $\tilde{\alpha}$ is strictly outer, then $\alpha$ has the Rohlin property and also is conjugate to $\id\otimes\gamma$, where $\gamma$ is the minimal action on $R_0$.
        \item If $\mod(\alpha)$ is faithful, then $\alpha$ has the Rohlin property.
    \end{enumerate}
\end{prop}
\begin{proof}
    (1):
    According to \cite[Theorem 2.4]{MT10}, it suffices to show that the dual action $\hat{\alpha}\colon\hat{G}\curvearrowright M\rtimes_{\alpha}G$ is approximately inner and centrally free.
    Note that $M\rtimes_{\alpha}G$ is a factor.
    Indeed, we have
    \[
    \Z(M\rtimes_{\alpha}G)\subset\Z(\tilde{M\rtimes_{\alpha}G}) =\Z(\tilde{M}\rtimes_{\tilde{\alpha}}G)\subset\Z(\tilde{M})
    \]
    and $\tilde{M}\cap(M\rtimes_{\alpha}G)=M$ since every element in $\tilde{M}\cap (M\rtimes_{\alpha}G)$ is fixed under $\tilde{\hat{\alpha}}=\hat{\tilde{\alpha}}$ and the trace-scaling flow $\theta$.

    For any $\pi\in\hat{G}$, $\tilde{\hat{\alpha}}_{\pi}$ is trivial on $\Z(\tilde{M\rtimes_{\alpha}G})\subset\Z(\tilde{M})$.
    Thus, $\tilde{\hat{\alpha}}_{\pi}$ has the trivial Connes--Takesaki module and $\hat{\alpha}_{\pi}$ is approximately inner by \cite[Lemma 7.5, Theorem 7.6(1)]{MT10}.
    
    Suppose $\hat{\alpha}_{\pi}$ is centrally trivial.
    From \cite[Theorem 7.6(2)]{MT10}, it follows that $\tilde{\hat{\alpha}}_{\pi}$ is inner, that is, there exists a unitary $W\in(\tilde{M\rtimes_{\alpha}G})\otimes\mathbb{B}(\mathcal{H}_{\pi})$ such that $\tilde{\hat{\alpha}}_{\pi}=W(\cdot\otimes 1)W^{\ast}$.
    Since $\tilde{\hat{\alpha}}_{\pi}(x)=x\otimes 1$ for all $x\in \tilde{M}$, the unitary $W$ belongs to $(\tilde{M}'\cap(\tilde{M}\rtimes_{\tilde{\alpha}}G))\otimes\mathbb{B}(\mathcal{H}_{\pi})=\Z(\tilde{M})\otimes\mathbb{B}(\mathcal{H}_{\pi})$ by the strict outerness of $\tilde{\alpha}$.
    But $\tilde{\hat{\alpha}}_{\pi}(\lambda_g)=\lambda_g\otimes\pi_g\neq W(\lambda_g\otimes 1)W^{\ast}$ because $W$ is fixed under $\tilde{\alpha}_g\otimes 1$ if $\mod(\alpha_g)=1$.
    This is a contradiction, and hence we have $\hat{\alpha}_{\pi}$ is centrally non-trivial.

    (2):
    By \cite[Corollary 5.10]{Iz03}, such an action $\alpha$ is unique up to conjugacy.
    Hence $\alpha$ and $\alpha\otimes\gamma$ are conjugate and they have the Rohlin property.
\end{proof}

% \begin{thm}
%     Let $\alpha\colon G\curvearrowright M$ be an action of a compact group $G$ on an AFD factor $M$.
%     Then the following conditions are equivalent:
%     \begin{enumerate}
%         \item $\tilde{\alpha}$ is strictly outer;
%         \item $\tilde{\alpha}|_H$ is strictly outer, where $H=\ker(\mod(\alpha))$;
%         \item $\alpha$ has the Rohlin property.
%     \end{enumerate}
% \end{thm}

\subsection{Examples and applications}

% \begin{exam}[non McDuff example]
%     Similar to \cite{Sh14}, there is an example of Rohlin action on a non-McDuff factor $G\curvearrowright L^{\infty}(G)^{\otimes\infty}\rtimes(S_{\infty}\ast S_{\infty})$.
% \end{exam}

\begin{exam}
    Let $\varphi$ be a faithful normal state on the AFD $\III_1$ factor $R_{\infty}$.
    Restrict the modular action $\sigma^{\varphi}$ to a countable subgroup $\Gamma<\mathbb{R}$, and we have a crossed product $M=R_{\infty}\rtimes_{\sigma^{\varphi}}\Gamma$.
    The dual action $\hat{\Gamma}\curvearrowright M$ has the Rohlin property since it has a faithful Connes--Takesaki module.
\end{exam}

\begin{exam}[non-example]
    Let $M$ be a factor of type $\III$, and $\varphi$ an almost periodic state, namely, $G\coloneqq\overline{\sigma^{\varphi}(\mathbb{R})}\subset\Aut(M)$ is compact.
    Then $G\curvearrowright M$ does not have the Rohlin property.
\end{exam}
\begin{proof}
    Any modular actions are centrally trivial.
\end{proof}

Let $N\subset M$ be an irreducible inclusion of $\II_1$ factors.
For $t>0$, we define the amplified inclusion by $N^t=p(N\otimes\mathbb{M}_n)p\subset p(M\otimes\mathbb{M}_n)p=M^t$, where $n\ge t$ is an integer and $p\in N\otimes\mathbb{M}_n$ is a projection such that $(\tau\otimes\Tr)(p)=t$.
The isomorphism class of the inclusion $N^t\subset M^t$ does not depend on the choice of $n$ and $p$.
The \emph{relative fundamental group} is defined by $\mathcal{F}(N\subset M)=\{t\in\mathbb{R}_+^{\ast}:(N\subset M)\cong (N^t\subset M^t)\}$.

If $N=M^{\beta}$ for some minimal action $\beta\colon G\curvearrowright M$, then we also consider the reduced action $\beta^t\colon G\curvearrowright M^t$, that is, $\beta_g^t(x)=(\beta_g\otimes\id)(x)$ for $x\in M^t=p(M\otimes\mathbb{M}_n)p$.
This depends only on $t>0$ up to conjugacy.

\begin{cor}
    Let $\beta\colon G\curvearrowright M$ be a Rohlin action of a compact group on a McDuff factor of type $\II_1$.
    Then $\mathcal{F}(M^{\beta}\subset M)=\mathbb{R}_+^{\ast}$.
    Moreover, for any $t>0$ there exists an isomorphism $\sigma\colon M\to M^t$ such that $\sigma\circ\beta^t\circ\sigma^{-1}=\beta$.
\end{cor}
\begin{proof}
    Let $\gamma\colon G\curvearrowright R$ be the minimal action.
    By the uniqueness theorem, for any $t>0$ there exists an isomorphism $\sigma\colon R\to R^t$ such that $\sigma\circ\gamma^t\circ\sigma^{-1}=\gamma$.
    In general, consider $\beta\otimes\gamma$ instead of $\beta$, and $\id\otimes\sigma$ intertwines so that $(\beta\otimes\gamma)^t=\beta\otimes\gamma^t\sim\beta\otimes\gamma$.
\end{proof}

\begin{cor}
    Let $\alpha$ and $\beta$ be Rohlin actions of a compact group $G$ on an AFD factor $M$.
    Then $\alpha$ and $\beta$ are conjugate if and only if $\mod(\alpha)=\mod(\beta)$.
\end{cor}
\begin{proof}
    By \cite[Theorem 1(i)]{KST92}, the condition $\mod(\alpha)=\mod(\beta)$ implies that $\beta_g\circ\alpha_g^{-1}\in\clInt(M)$ for all $g\in G$.
    Thus, Theorem \ref{thm:classification} completes the proof.
\end{proof}

\section{Cohomology vanishing}

\subsection{1-cohomology vanishing}

There are several known results on the vanishing of 1-coho\-mology.
For instance, Connes' $2\times 2$ matrix trick is effective for minimal actions of compact groups on finite von Neumann algebras.
See \cite[Proposition 5.2]{Iz03} for a general treatment.
For Rohlin flows and actions, the approximate vanishing of 1-cohomology is also known.
In contrast, for compact group actions, we establish the strict vanishing of 1-cohomology for Rohlin actions.

\begin{thm}
    Let $G$ be a compact group and $\alpha\colon G\curvearrowright M$ be an action on a factor $M$.
    If $\alpha$ has the Rohlin property, then every $\alpha$-cocycle is a coboundary.
\end{thm}
\begin{proof}
    Let $v$ be an $\alpha$-cocycle.
    We fix a faithful normal state $\varphi$ on $M$.
    First we claim that for every $\varepsilon>0$ and unitary $\tilde{w}\in\U(M)$, there exists $w\in\U(M)$ such that
    \[\int_G\|\tilde{w}(v_g\alpha_g(w)-w)\|_{\varphi}^{\sharp 2}\,dg<\varepsilon,\quad \|\tilde{w}(w-1)\|_{\varphi}^{\sharp 2}<\int_G\|\tilde{w}(v_g-1)\|_{\varphi}^{\sharp 2}\,dg+\varepsilon.\]
    Let $W\in M\otimes L^{\infty}(G)\subset M_{\alpha}^{\omega}$ be given by $W(h)=v_h$.
    Then we have $[v_g\alpha_g(W)](h)=v_g\alpha_g(W(g^{-1}h))=v_h=W(h)$ and hence $v_g\alpha_g(W)=W$, 
    \[\|\tilde{w}(W-1)\|_{\varphi^{\omega}}^{\sharp 2}=\int_G\|\tilde{w}(v_g-1)\|_{\varphi}^{\sharp 2}\,dg.\]
    Take a unitary representing sequence $(w^{\nu})_{\nu}$ of $W$ and the claim follows for some $w=w^{\nu}$ by Lemma \ref{lem:limit}.

    Next, from the above claim, we inductively construct a sequence of unitaries $w^{(n)}$, by letting $v^{(0)}=v$, $w^{(0)}=1$, $v_g^{(n)}=(w^{(n-1)})^{\ast}v_g^{(n-1)}\alpha_g(w^{(n-1)})$ and $\tilde{w}^{(n-1)}=w^{(0)}w^{(1)}\cdots w^{(n-1)}=\tilde{w}^{(n-2)}w^{(n-1)}$.
    We choose a unitary $w^{(n)}$ such that
    \[
    \int_G\|\tilde{w}^{(n-1)}(v_g^{(n)}\alpha_g(w^{(n)})-w^{(n)})\|_{\varphi}^{\sharp 2}\,dg<2^{-n-1},
    \]\[
    \|\tilde{w}^{(n-1)}(w^{(n)}-1)\|_{\varphi}^{\sharp 2}<\int_G\|\tilde{w}^{(n-1)}(v_g^{(n)}-1)\|_{\varphi}^{\sharp 2}\,dg+2^{-n}.
    \]
    Then the sequence $(\tilde{w}^{(n)})_n$ converges to some unitary $w\in\U(M)$.
    Indeed, we have 
    \begin{align*}
        \|\tilde{w}^{(n+1)}-\tilde{w}^{(n)}\|_{\varphi}^{\sharp 2} &= \|\tilde{w}^{(n)}(w^{(n+1)}-1)\|_{\varphi}^{\sharp 2}\\
        &< \int_G\|\tilde{w}^{(n)}(v_g^{(n+1)}-1)\|_{\varphi}^{\sharp 2}\,dg+2^{-n-1}\\
        &= \int_G\|\tilde{w}^{(n-1)}(v_g^{(n)}\alpha_g(w^{(n)})-w^{(n)})\|_{\varphi}^{\sharp 2}\,dg+2^{-n-1}<2^{-n}.
    \end{align*}
    Hence we see that $\tilde{w}^{(n-1)}(v_g^{(n)}\alpha_g(w^{(n)})-w^{(n)})=v_g\alpha_g(\tilde{w}^{(n)})-\tilde{w}^{(n)}$ tends to $v_g\alpha_g(w)-w$, and this is exactly zero.
\end{proof}

\subsection{Cocycle actions}

Sutherland has shown that any cocycle action on a properly infinite von Neumann algebra is a unitary perturbation of a genuine action.
Specifically, every 2-cohomology vanishes in the following cases.

\begin{prop}[{\cite[Theorem 4.1.3]{Su80}}]
    Let $(\alpha,u)\colon G\curvearrowright M$ be a cocycle action.
    If either $M$ is properly infinite, or $G$ is finite and $M$ is of type $\II_1$, then $u$ is a coboundary.
\end{prop}

In this subsection, we always assume that $M$ is a $\II_1$ factor and let $\|\cdot\|_2$ denote the canonical 2-norm on $M$ and $M^{\omega}$.
Recall that a cocycle action $(\alpha,u)\colon G\curvearrowright M$ on a factor is called \emph{strictly outer} if $M'\cap(M\rtimes_{\alpha,u}G)=\mathbb{C}$ holds, in analogy with actions.

\begin{prop}
    If a cocycle action $(\alpha,u)\colon G\curvearrowright M$ has the Rohlin property, then $u$ is a coboundary.
\end{prop}
\begin{proof}
    First we claim that for any $\varepsilon>0$, there exists a Borel map $v\colon G\to\U(M)$ such that
    \[
    \iint_{G\times G}\|v_g\alpha_g(v_h)u_{g,h}v_{gh}^{\ast}-1\|_2^2\,dg\,dh < \varepsilon,\quad
    \int_G\|v_g-1\|_2^2\,dg < \iint_{G\times G}\|u_{g,h}-1\|_2^2\,dg\,dh+\varepsilon.
    \]
    To prove this, we define a map $V\colon G\to\U(M\otimes L^{\infty}(G))\subset\U(M^{\omega}_{\alpha})$ by $V_g(h)=u_{g,g^{-1}h}^{\ast}$.
    By the 2-cocycle identity, it follows that
    \[
    [V_g\alpha_g(V_h)u_{g,h}V_{gh}^{\ast}](k) = u_{g,g^{-1}k}^{\ast}\alpha_g(u_{h,h^{-1}g^{-1}k}^{\ast})u_{g,h}u_{gh,h^{-1}g^{-1}k}=1
    \]
    and
    \[
    \int_G\|V_g-1\|_2^2\,dg = \iint_{G\times G}\|u_{g,h}-1\|_2^2\,dg\,dh.
    \]
    Take a representing sequence $(v^{\nu})_{\nu}$ of $V$ by Lemma \ref{lem:Borellift}, and put $v=v^{\nu}$ for suitable $\nu$.
    This proves the claim.

    Now we can find a sequence of Borel maps $v^{(n)}\colon G\to\U(M)$ such that if we put $\alpha_g^{(0)}=\alpha_g$, $u_{g,h}^{(0)}=u_{g,h}$, $\alpha_g^{(n+1)}=\Ad(v_g^{(n)})\circ\alpha_g^{(n)}$ and $u_{g,h}^{(n+1)}=v_g^{(n)}\alpha_g^{(n)}(v_h^{(n)})u_{g,h}^{(n)}(v_{gh}^{(n)})^{\ast}$, then
    \begin{align*}
        \iint_{G\times G}\|v_g^{(n)}\alpha_g^{(n)}(v_h^{(n)})u_{g,h}^{(n)}(v_{gh}^{(n)})^{\ast}-1\|_2^2\,dg\,dh &< 2^{-n-3},\\
        \int_G\|v_g^{(n)}-1\|_2^2\,dg < \iint_{G\times G}\|u_{g,h}^{(n)}-1\|_2^2\,dg\,dh+2^{-n-2} &< 2^{-n-1}.
    \end{align*}
    The second inequality means that $\lim_{n\to\infty}v^{(n)}\cdots v^{(2)}v^{(1)}$ converges to some Borel unitary map $v\colon G\to\U(M)$.
    We can check that $v_g\alpha_g(v_h)u_{g,h}v_{gh}^{\ast}=1$ and this completes the proof.
\end{proof}

\begin{rem}
    Furthermore, if the 2-cocycle $u$ is close to 1, then we can choose a cochain $v$ staying close to 1.
    See \cite[Theorem 5.5]{MT16}.
\end{rem}

As a conclusion, we obtain the following theorem.

\begin{thm}
    Let $(\alpha,u)\colon G\curvearrowright R$ be a cocycle action of a compact group $G$ on the AFD $\II_1$ factor.
    If $(\alpha,u)$ is strictly outer, then $u$ is a coboundary.
\end{thm}
\begin{proof}
    Let $Q$ be a separable factor of type $\I_{\infty}$.
    Consider $(\alpha\otimes\id,u\otimes 1)\colon G\curvearrowright R\otimes Q$.
    Since $R\otimes Q$ is properly infinite, the cocycle $u\otimes 1$ is a coboundary and $(\alpha\otimes\id,u\otimes 1)$ is perturbed to some strictly outer (hence minimal) action.
    Recall that such a minimal action on the AFD $\II_{\infty}$ factor $R\otimes Q$ is unique up to conjugacy (Theorem \ref{thm:uniqmin}).
    Therefore $(\alpha\otimes\id,u\otimes 1)$ has the Rohlin property.

    Since $(R\otimes Q)_{\omega,\alpha}=R_{\omega,\alpha}$ contains $L^{\infty}(G)$ equivariantly, it then follows that $(\alpha,u)$ also has the Rohlin property.
    The last proposition completes the proof.
\end{proof}

\begin{cor}
    Let $\alpha\colon G\curvearrowright M$ be a Rohlin action of a compact group $G$ on a factor $M$ and $c\in\mathrm{Z}^2(G;\mathbb{T})$ a 2-cocycle.
    Then the cocycle crossed product $M\rtimes_{\alpha,c}G$ is isomorphic to $M\rtimes_{\alpha}G$.
\end{cor}


\begin{thebibliography}{KLW23}

\bibitem[AH14]{AH14} H. Ando, U. Haagerup, \textit{Ultraproducts of von Neumann algebras}, J. Funct. Anal. \textbf{266} (2014), no.~12, 6842--6913.

\bibitem[AM13]{AM13} T. Austin, C. C. Moore, \textit{Continuity properties of measurable group cohomology}, Math. Ann. \textbf{356} (2013), no.~3, 885--937.

\bibitem[BK24]{BK24} M. Bischoff, P. Karmakar, \textit{Anomalies for conformal nets associated with lattices and $T$-kernels}, preprint, \texttt{arXiv:2406.09667} (2024).

\bibitem[Co77]{Co77} A. Connes, \textit{Periodic automorphisms of the hyperfinite factor of type $\II_1$}, Acta Sci. Math. (Szeged) \textbf{39} (1977), no.~1-2, 39--66.

\bibitem[De74]{De74} P. Deligne, \textit{Th\'eorie de {H}odge}. III, Inst. Hautes \'Etudes Sci. Publ. Math. No.~\textbf{44} (1974), 5--77.

\bibitem[FT01]{FT01} T. Falcone, M. Takesaki, \textit{The non-commutative flow of weights on a von Neumann algebra}, J. Funct. Anal. \textbf{182} (2001), no.~1, 170--206.

\bibitem[Ga17]{Ga17} E. Gardella, \textit{Crossed products by compact group actions with the Rokhlin property}, J. Noncommut. Geom. \textbf{11} (2017), no.~4, 1593--1626.

\bibitem[HW07]{HW07} I. Hirshberg, W. Winter, \textit{Rokhlin actions and self-absorbing C*-algebras}, Pacific J. Math. \textbf{233} (2007), no.~1, 125--143.

\bibitem[Iz03]{Iz03} M. Izumi, \textit{Canonical extension of endomorphisms of type $\III$ factors}, Amer. J. Math. \textbf{125} (2003), no.~1, 1--56.

\bibitem[Iz04a]{Iz04a} M. Izumi, \textit{Finite group actions on C*-algebras with the Rohlin property}. I, Duke Math. J. \textbf{122} (2004), no.~2, 233--280.

\bibitem[Iz04b]{Iz04b} M. Izumi, \textit{Finite group actions on C*-algebras with the Rohlin property}. II, Adv. Math. \textbf{184} (2004), no.~1, 119--160.

\bibitem[Je25]{Je25} D. Jekel, \textit{The unitary group of a $\II_1$ factor is SOT-contractible}, Math. Ann. (2025), 1--9.

\bibitem[Jo79]{Jo79} V. F. R. Jones, \textit{An invariant for group actions}, Alg\`ebres d'op\'erateurs, Lecture Notes in Math. \textbf{725} (1979), 237--253.

\bibitem[JT84]{JT84} V. F. R. Jones, M. Takesaki, \textit{Actions of compact abelian groups on semifinite injective factors}, Acta Math. \textbf{153} (1984), no.~3-4, 213--258.

\bibitem[KST92]{KST92} Y. Kawahigashi, C. E. Sutherland, M. Takesaki, \textit{The structure of the automorphism group of an injective factor and the cocycle conjugacy of discrete abelian group actions}, Acta Math. \textbf{169} (1992), no.~1-2, 105--130.

\bibitem[KT92]{KT92} Y. Kawahigashi, M. Takesaki, \textit{Compact abelian group actions on injective factors}, J. Funct. Anal. \textbf{105} (1992), no.~1, 112--128.

\bibitem[Ka00]{Ka00} K. Kawamuro, \textit{A Rohlin property for one-parameter automorphism groups of the hyperfinite $\II_1$ factor}, Publ. RIMS. \textbf{36} (2000), no.~5, 641--657.

\bibitem[Ki96]{Ki96} A. Kishimoto, \textit{A Rohlin property for one-parameter automorphism groups}, Comm. Math. Phys. \textbf{179} (1996), no.~3, 599--622.

\bibitem[KLW23]{KLW23} P. Kristel, M. Ludewig, K. Waldorf, \textit{A representation of the string 2-group}, preprint, \texttt{arXiv:2308.05139} (2023).

\bibitem[Ma25]{Ma25} A. Marrakchi, \textit{Almost almost periodic type $\III_1$ factors and their 3-cohomology obstructions}, preprint, \texttt{arXiv:2502.01516} (2025).

\bibitem[MT07]{MT07} T. Masuda, R. Tomatsu, \textit{Classification of minimal actions of a compact Kac algebra with amenable dual}, Comm. Math. Phys. \textbf{274} (2007), no.~2, 487--551.

\bibitem[MT10]{MT10} T. Masuda, R. Tomatsu, \textit{Classification of minimal actions of a compact Kac algebra with amenable dual on injective factors of type $\III$}, J. Funct. Anal. \textbf{258} (2010), no.~6, 1965--2025.

\bibitem[MT16]{MT16} T. Masuda, R. Tomatsu, \textit{Rohlin flows on von Neumann algebras}, Mem. Amer. Math. Soc. \textbf{244} (2016), no.~1153.

\bibitem[MT17]{MT17} T. Masuda, R. Tomatsu, \textit{Classification of actions of discrete Kac algebras on injective factors}, Mem. Amer. Math. Soc. \textbf{245} (2017), no.~1160, ix+118 pp.

\bibitem[Mo64]{Mo64} C. C. Moore, \textit{Extensions and low dimensional cohomology theory of locally compact groups}. I,II, Trans. Amer. Math. Soc. \textbf{113} (1964), 40--63; ibid., 64--86.

\bibitem[Mo76]{Mo76} C. C. Moore, \textit{Group extensions and cohomology for locally compact groups}. III,IV, Trans. Amer. Math. Soc. \textbf{221} (1976), no.~1, 1--33; ibid., 35--58.

\bibitem[NSW13]{NSW13} T. Nikolaus, C. Sachse, C. Wockel, \textit{A smooth model for the string group}, Int. Math. Res. Not. IMRN (2013), no.~16, 3678--3721.

\bibitem[Oc85]{Oc85} A. Ocneanu, \textit{Actions of Discrete Amenable Groups on Von Neumann Algebras}, Lecture notes in Math. \textbf{1138}, Springer-Verlag, 1985.

\bibitem[Po21]{Po21} S. Popa, \textit{On the vanishing cohomology problem for cocycle actions of groups on $\II_1$ factors}, Ann. Sci. \'Ec. Norm. Sup\'er. (4) \textbf{54} (2021), no.~2, 407--443.

\bibitem[Sh14]{Sh14} K. Shimada, \textit{Locally compact separable abelian group actions on factors with the Rokhlin property}, Publ. RIMS. \textbf{50} (2014), no.~3, 363--381.

\bibitem[St96]{St96} S. Stolz, \textit{A conjecture concerning positive Ricci curvature and the Witten genus}, Math. Ann. \textbf{304} (1996), no.~4, 785--800.

% \bibitem[ST04]{ST04} S. Stolz, P. Teichner, \textit{``What is an elliptic object?''}, London Math. Soc. Lecture Note Ser. \textbf{308} (2004), 247--343.

\bibitem[Su80]{Su80} C. E. Sutherland, \textit{Cohomology and extensions of von Neumann algebras}. II, Publ. RIMS, \textbf{16} (1980), 135--174.

\bibitem[Ta03]{Ta03} M. Takesaki, \textit{Theory of Operator Algebras II}, Encyclopedia Math. Sci. \textbf{125}, Springer-Verlag, Berlin, 2003.

\bibitem[Tu06]{Tu06} J. L. Tu, \textit{Groupoid cohomology and extensions}, Trans. Amer. Math. Soc. \textbf{358} (2006), no.~11, 4721--4747.

\bibitem[Va01]{Va01} S. Vaes, \textit{The unitary implementation of a locally compact quantum group action}, J. Funct. Anal. \textbf{180} (2001), no.~2, 426--480.

\bibitem[VV15]{VV15} S. Vaes, P. Verraedt, \textit{Classification of type $\III$ Bernoulli crossed products}, Adv. Math. \textbf{281} (2015), 296--332.

\bibitem[Wa06]{Wa06} A. Wassermann, \textit{Operator algebras and elliptic cohomology}, Trends in noncommutative geometry, NCGW03, 18 December 2006 (handwritten manuscript).

\end{thebibliography}
\end{document}